\documentclass[12pt,a4paper]{article}
\usepackage[english]{babel}

\usepackage[top=2.8cm,bottom=2.8cm,left=2.8cm,right=2.8cm]{geometry}

\usepackage[utf8]{inputenc}
\usepackage{t1enc}
\usepackage{amsfonts}
\usepackage{amsmath}
\usepackage{color}
\usepackage{amssymb}
\usepackage{amsthm}
\usepackage{graphicx}
\usepackage{footnote}
\usepackage{epstopdf}
\usepackage{tikz}
\usepackage{caption}
\usetikzlibrary{arrows}

\title{Ef{\kern0pt}f{\kern0pt}iciency analysis of double perturbed pairwise comparison matrices}
\date{}

\newtheorem{innercustomlemma}{Lemma}
\newenvironment{customlemma}[1]
  {\renewcommand\theinnercustomlemma{#1}\innercustomlemma}
  {\endinnercustomlemma}

\newtheorem{remark}{Remark} %vagy számozzuk inkább? akkor ki kell venni a *-ot
\newtheorem{theorem}{Theorem}
\newtheorem{corollary}{Corollary}
\newtheorem{lemma}{Lemma}
\newtheorem{proposition}{Proposition}
\newtheorem{question}{Open problem}

\theoremstyle{definition}
\newtheorem{example}{Example}
\newtheorem{definition}{Definition}

\DeclareMathOperator{\diag}{diag}
\DeclareMathOperator{\adj}{adj}

\begin{document}

\sloppy % A marg\'ora l\'og\'o sz\"ovegek elker\"uléséhez.
\allowdisplaybreaks % Hosszú align-okat elt\"orhet.

\title{Ef{\kern0pt}f{\kern0pt}iciency analysis of double perturbed pairwise comparison matrices}

\author{Krist\'of \'Abele-Nagy$^{\,\,1,2}$\thanks{\textit{E-mail: kristof.abele-nagy@uni-corvinus.hu}}
\\
\and S\'andor Boz\'oki$^{\,\,1,2}$\thanks{\textit{E-mail: bozoki.sandor@sztaki.mta.hu}} \\
\and Örs Rebák\thanks{\textit{E-mail: rebakors@gmail.com}}}

\maketitle

\footnotetext[1]{Laboratory on Engineering and Management
Intelligence, Research Group of Operations Research and Decision
Systems, Institute for Computer Science and Control, Hungarian
Academy of Sciences; Mail: 1518 Budapest, P.O.~Box 63, Hungary.}
\footnotetext[2]{Department of Operations Research and Actuarial Sciences, Corvinus
University
of Budapest, Hungary
}

\begin{abstract}
Ef{\kern0pt}f{\kern0pt}iciency is a core concept of multi-objective optimization problems and multi-attribute decision making. In the case of pairwise comparison matrices a weight vector is called ef{\kern0pt}f{\kern0pt}icient if the approximations of the elements of the pairwise comparison matrix made by the ratios of the weights cannot be improved in any position without making it worse in some other position. A pairwise comparison matrix is called double perturbed if it can be made consistent by altering two elements and their reciprocals. The most frequently used weighting method, the eigenvector method is analyzed in the paper, and it is shown that it produces an ef{\kern0pt}f{\kern0pt}icient weight vector for double perturbed pairwise comparison matrices.

\end{abstract}

\section{Introduction}
Ranking alternatives, or picking the best alternative is a commonly investigated problem. The case of a single cardinal objective function to be maximized or minimized is long studied by various operations research disciplines. This is however often not feasible. Alternatives can be ranked by assigning a cardinal utility to them, or by setting up ordinal preference relations among them. In the case of a single criterion and a single decision maker, modelling the preferences is often possible through standard methods. If there are multiple, often contradicting criteria, this becomes signif{\kern0pt}icantly harder. A dominant alternative, which is the best with respect to all criteria, very rarely exists. Thus, when a decision making method is used to aid the decision of a decision maker, some form of compromise is needed. Modelling the preferences of the decision maker by ranking or weighting the criteria can accomplish such a compromise. It allows the ``best'' alternative to be chosen (or the possible alternatives to be ranked) with respect to the subjective preferences of the decision maker. Examples of multi-criteria decision problems range from ``Which house to buy?'' or ``What should the company invest in?'' to public tenders.

When weighting criteria, giving the weights directly is almost never feasible. Instead, a common method is to apply pairwise comparisons. Answers to the questions ``How many times is Criterion A more important than Criterion B?'' and so on (which are explicit cardinal ratios) can be arranged in a matrix, called a pairwise comparison matrix (PCM). Formally, a PCM is a square matrix $\mathbf{A}=[a_{ij}]_{i,j=1,\dots,n}$ with the properties $a_{ij}>0$ and $a_{ij}=1/a_{ji}$ (which implies $a_{ii}=1$). If the cardinal transitivity property $a_{ik}a_{kj}=a_{ij}$ for all $i,j,k=1,\dots,n$ also holds for a PCM, it is called consistent, otherwise it is called inconsistent \cite{Saaty1977}. Let $\mathcal{PCM}_n$ denote the set of PCMs of size $n\times n$.
%A PCM is called consistent, if cardinal transitivity holds, namely if Criterion A is $p$ times more important than Criterion B, and Criterion B is $q$ times more important than Criterion C, then Criterion A is $pq$ times more important than Criterion C.
The next step is to extract the weights of criteria from the PCM.
Several methods exist for this task
\cite{BajwaChooWedley2008,ChooWedley2004,Dijkstra2013,GolanyKress1993}.
\color{black} The eigenvector method (EM) is one of the classical
\cite{Saaty1977}, important and most often studied weighting
methods related to pairwise comparison matrices, its further
analysis is actual and relevant both in decision theory and
operations research. We focus on EM in this paper. \color{black}
The eigenvector method gives the weight vector
$\mathbf{w}^{EM}=(w_1,\dots,w_n)^{T}$ as the right Perron
eigenvector of $\mathbf{A} \in \mathcal{PCM}_n$, thus
$\mathbf{A}\mathbf{w}^{EM}=\lambda_{\max} \mathbf{w}^{EM}$ holds,
where $\lambda_{\max}$ is the principal eigenvalue of
$\mathbf{A}$. $\lambda_{\max}\geq n$, and $\lambda_{\max} = n$ if
and only if $\mathbf{A}$ is consistent \cite{Saaty1977}. A
consistent PCM can be written as
\[
\mathbf{A}  =
\left(\begin{array}{ccccc}
1 & x_1 & x_2 & \dots & x_{n-1} \\ 1/x_1 & 1 & x_2/x_1 & \dots & x_{n-1}/x_1 \\
1/x_2 & x_1/x_2 & 1 & \dots & x_{n-1}/x_2 \\
\vdots & \vdots & \vdots & \ddots & \vdots \\
1/x_{n-1} & x_1/x_{n-1} & x_2/x_{n-1} & \dots & 1
\end{array} \right)\in \mathcal{PCM}_n,
\]
where $x_1,\dots,x_{n-1}>0$.

The elements of a PCM approximate the ratios of the weights,
therefore the ratios of the elements of the weight vector should
be as close as possible to the corresponding matrix elements. If a
weight vector cannot be trivially improved in this regard (there
is no other weight vector which is at least as good approximation,
and strictly better in at least one position), it is called Pareto
optimal or ef{\kern0pt}f{\kern0pt}icient. It has been proved that the eigenvector
method does not always produce an ef{\kern0pt}f{\kern0pt}icient solution
\color{black}\cite[Section
3]{BlanqueroCarrizosaConde2006}\color{black}. However, in some
special cases the eigenvector method always gives an ef{\kern0pt}f{\kern0pt}icient
weight vector. If the PCM is simple perturbed, i.e., it dif{\kern0pt}fers
from a consistent PCM in only one element and its reciprocal, the
principal right eigenvector is ef{\kern0pt}f{\kern0pt}icient
\cite{Abele-NagyBozoki2016}. In the paper this will be extended to
double perturbed PCMs, which only dif{\kern0pt}fer from a consistent PCM in
two elements and their reciprocals.

These special types of PCMs are not just theoretically important, but also occur in real decision problems. Poesz \cite{Poesz2009} gathered a handful of empirical PCMs that were analyzed in \cite{BozokiFulopPoesz2011}. \color{black} Although double perturbed PCMs are rare among large PCMs, they appear more frequently among smaller matrices, in other words, when the number of criteria is small. This is especially true if one considers simple perturbed and consistent PCMs as special cases of double perturbed PCMs (see \cite[Table 1]{BozokiFulopPoesz2011} -- note there is a misprint in the cited Table, the number in the ``3 elements to modify'' column in the $4\times 4$ row should be 6 instead of 0). We also conducted an analysis of the prevalence of double perturbed PCMs among the empirical matrices analyzed in \cite{BozokiDezsoPoeszTemesi2013}. Below is a supplemented version of \cite[Table 1]{BozokiFulopPoesz2011} that also includes the results from \cite{BozokiDezsoPoeszTemesi2013}. The PCMs in \cite{BozokiDezsoPoeszTemesi2013} had sizes $4\times 4$, $6\times 6$ and $8\times 8$. %The numbers before the $+$ sign are from \cite{BozokiFulopPoesz2011}, the numbers after the $+$ sign are from \cite{BozokiDezsoPoeszTemesi2013}. Where there is no $+$ sign, the numbers are from \cite{BozokiFulopPoesz2011}.
The numbers after the $+$ sign are from \cite{BozokiDezsoPoeszTemesi2013}, all others are from \cite{BozokiFulopPoesz2011}. As it can be seen from Table \ref{konz_tablazat}, double (and less) perturbed PCMs are rare among large matrices, but they appear among smaller ones.

\begin{center}
\begin{table}
\centering \captionsetup{justification=centering,margin=0.8cm}
\color{black}
\begin{tabular}{|l|c|c|c|c|c|}
\hline
           & Number of     &            & 1 element & 2 elements & 3 elements     \\
Dimension  &      matrices & Consistent & to modify & to modify  & to modify      \\
\hline
$3\times 3$&         30    &      14    &    16     &      --    &      --        \\
\hline
$4\times 4$&         20+134    &      1+3     &    6+10      &       7+31    &      6*+90         \\
\hline
$5\times 5$&         19    &      1     &    1      &       5    &      1         \\
\hline
$6\times 6$&         21+152    &      0+1     &    1+1      &      1+1     &      0+0         \\
\hline
$7\times 7$&               &            &           &            &                \\
and larger &         47+159   &      0+0     &    1+0      &      0+0     &      0+0         \\
\hline
\end{tabular}
\caption{\small \color{black} The number of element modif{\kern0pt}ications
needed to get a consistent PCM. \newline *There was a misprint in
\cite[Table 1]{BozokiFulopPoesz2011}, this number was 0. The
correct value is 6.} \label{konz_tablazat}
\end{table}
\end{center}
\color{black} %out of 90 matrices of size at most $6 \times 6$, 53 were either consistent, simple perturbed or double perturbed (see \cite[Table 1]{BozokiFulopPoesz2011}).

In Section 2 we will introduce the key def{\kern0pt}initions and tools used in the paper, together with an example. In Section 3 the main result of the paper is presented: through obtaining explicit formulas for the principal right eigenvector and a series of lemmas, the ef{\kern0pt}f{\kern0pt}iciency of the principal right eigenvector is shown for the case of double perturbed PCMs. The proofs of the lemmas, are given in detail in the Appendix. In Section 4 conclusions follow.

\section{Ef{\kern0pt}f{\kern0pt}iciency and perturbed pairwise comparison matrices}

\color{black} The general form of a multi-objective optimization
problem (\cite[Chapter 2]{Ehrgott2000}\cite[Chapter
6]{Steuer1986}) is
\begin{eqnarray*}
&\min \{ f_1(\mathbf{y}), f_2(\mathbf{y}), \ldots, f_m(\mathbf{y}), \ldots, f_M(\mathbf{y})\}   \\
&\text{subject to } \mathbf{y} \in S
\end{eqnarray*}
where $M \geq 2$ denotes the number of objective functions,
$f_m: \mathbb{R}^n \rightarrow \mathbb{R}$ for all $1 \leq m \leq M.$
Variables are $\mathbf{y} = (y_1, y_2, \ldots, y_n)$ and the feasible set is denoted by $S \subseteq \mathbb{R}^n$.

Ef{\kern0pt}f{\kern0pt}iciency or Pareto optimality is a basic concept of multi-objective optimization
and multi-attribute decision making, too.
A vector $\mathbf{y} \in S$ is called ef{\kern0pt}f{\kern0pt}icient, if there does not exist another
vector $\mathbf{y}^{\prime} \in S$  such that
$f_m(\mathbf{y}^{\prime}) \leq f_m(\mathbf{y}) $ for all $1 \leq m \leq M,$
and $f_k(\mathbf{y}^{\prime}) < f_k(\mathbf{y}) $ for at least one index $k$.

Let $\mathbf{A} =
\left[
a_{ij}
\right]_{i,j=1,\ldots,n} \in \mathcal{PCM}_n$  and
$\mathbf{w} = (w_1, w_2, \ldots, w_n)^{T}$ be a positive weight vector
($S = \mathbb{R}^n_{++},$ the positive orthant of the $n$-dimensional Euclidean space),
where $n$ is the number of criteria. Let us specify the
objective functions by $f_{ij}(\mathbf{w}) := \left| a_{ij} - \frac{w_i}{w_j} \right|$
for all $i \neq j$. We have $M = n^2-n$ objective functions.
\color{black}

\begin{definition} \label{def:DefinitionEfficient}  % \ref{def:DefinitionEfficient}
 A positive weight vector  $\mathbf{w}$ is called \emph{ef{\kern0pt}f{\kern0pt}icient}
if no other positive weight vector
$\mathbf{w^{\prime}} = (w^{\prime}_1, w^{\prime}_2, \ldots, w^{\prime}_n)^{T}$
exists such that
\begin{align}
 \left|a_{ij} - \frac{w^{\prime}_i}{w^{\prime}_j} \right| &\leq \left|a_{ij} - \frac{w_i}{w_j} \right| \qquad \text{ for all } 1 \leq i,j \leq n,  \label{eqn:eff1}\\
 \left|a_{k{\ell}} - \frac{w^{\prime}_k}{w^{\prime}_{\ell}} \right| &<  \left|a_{k{\ell}} - \frac{w_k}{w_{\ell}} \right|  \qquad \text{ for some } 1 \leq k,\ell \leq n. \label{eqn:eff2}
\end{align}
\end{definition}

 A weight vector
$\mathbf{w}$ is called \emph{inef{\kern0pt}f{\kern0pt}icient} if it is not ef{\kern0pt}f{\kern0pt}icient.

\color{black} It follows from the def{\kern0pt}inition that an arbitrary
renormalization does not inf{\kern0pt}luence (in)ef{\kern0pt}f{\kern0pt}iciency.

\begin{remark}
Weight vector $\mathbf{w}$ is ef{\kern0pt}f{\kern0pt}icient if and only if $c\mathbf{w}$ is ef{\kern0pt}f{\kern0pt}icient for any $c > 0$.
\end{remark}
\color{black}

For a consistent PCM $a_{ij}= w^{EM}_i/w^{EM}_j$ for all $i,j=1,\dots,n$ \cite{Saaty1977}, which implies the following remark:
\begin{remark}
\label{ConsistentEfficient}
The principal right eigenvector $\mathbf{w}^{EM}$ is ef{\kern0pt}f{\kern0pt}icient for every consistent PCM.
\end{remark}

For inconsistent PCMs however, the principal right eigenvector can be inef{\kern0pt}f{\kern0pt}icient, found by Blanquero, Carrizosa and Conde \cite[Section 3]{BlanqueroCarrizosaConde2006}. This result was also reinforced by Bajwa, Choo and Wedley \cite{BajwaChooWedley2008}, by Conde and P\'erez \cite{CondePerez2010} and by Fedrizzi \cite{Fedrizzi2013}. Blanquero, Carrizosa and Conde \cite{BlanqueroCarrizosaConde2006} developed LP models to test whether a weight vector is ef{\kern0pt}f{\kern0pt}icient. Boz\'oki and F\"ul\"op \cite{BozokiFulop2015} further developed the models and provided algorithms to improve an inef{\kern0pt}f{\kern0pt}icient weight vector. Anholcer and F\"ul\"op \cite{AnholcerFulop2015} devised a new algorithm to derive an ef{\kern0pt}f{\kern0pt}icient solution from an inconsistent PCM.

Furthermore, Boz\'oki \cite{Bozoki2014} showed that the principal right eigenvector of a whole class of matrices, namely the parametric PCM
\[
\mathbf{A}(p,q) =
\begin{pmatrix}
   1     &     p    &     p    &   p     &  \ldots &    p    &  p   \\
 1/p     &     1    &     q    &   1     &  \ldots &    1    &  1/q     \\
 1/p     &    1/q   &     1    &   q     &  \ldots &    1    &  1       \\
  \vdots &   \vdots &   \vdots &  \ddots &         &  \vdots &   \vdots  \\
  \vdots &   \vdots &   \vdots &         &  \ddots &  \vdots &   \vdots  \\
 1/p     &     1    &     1    &   1     &  \ldots &    1    &   q       \\
 1/p     &     q    &     1    &   1     &  \ldots &    1/q    &   1
\end{pmatrix}\in \mathcal{PCM}_n,
\]
where $n \geq 4,$ $p > 0$ and $1 \neq q > 0,$ is inef{\kern0pt}f{\kern0pt}icient.

Several necessary and suf{\kern0pt}f{\kern0pt}icient conditions were examined by Blanquero, Carrizosa and Conde \cite{BlanqueroCarrizosaConde2006}, one of which is of crucial importance here. It uses a directed graph representation as follows:

\begin{definition} \label{def:DirectedGraphEfficient}  % \ref{def:DirectedGraphEfficient} %MÁSOLVA!!!!!!!!!!!!!!!
Let $\mathbf{A} =
\left[
a_{ij}
\right]_{i,j=1,\ldots,n} \in \mathcal{PCM}_n$  and
$\mathbf{w} = (w_1, w_2, \ldots, w_n)^{T}$ be a positive weight vector.
 A directed graph $G=(V,\overrightarrow{E})_{\mathbf{A},\mathbf{w}}$ is def{\kern0pt}ined as follows:
$V=\{1,2,\ldots,n\} $ and
\[
 \overrightarrow{E} = \left\{ \text{arc(}  i \rightarrow j\text{)}   \left|  \frac{w_i}{w_j} \geq a_{ij}, i \neq j \right. \right\}.
\]
\end{definition}
It follows from Def{\kern0pt}inition \ref{def:DirectedGraphEfficient} that if $w_i/w_j=a_{ij}$, then there is a bidirected arc between nodes $i$ and $j$.
The result of Blanquero, Carrizosa and Conde using this representation is as follows:
\begin{theorem}[{\cite[Corollary 10]{BlanqueroCarrizosaConde2006}}] \label{thm:TheoremDirectedGraphEfficient} %MÁSOLVA!!!!!!!!!!!!!!!!!!!!!!!!!!!!!!!!!!!!!!!!!!!!!!!
Let $\mathbf{A} \in \mathcal{PCM}_n$.
A weight vector $\mathbf{w}$ is ef{\kern0pt}f{\kern0pt}icient if and only if
$G=(V,\overrightarrow{E})_{\mathbf{A},\mathbf{w}}$
is  a strongly connected digraph, that is,
there exist directed paths from $i$ to $j$ and from $j$ to $i$ for all
pairs of  nodes  $i , j$.
\end{theorem}

The following numerical example provides an illustration for Theorem \ref{thm:TheoremDirectedGraphEfficient}.

%PÉLDA!!!!!!!!!!!!!!!!!!!!!!!!!!!!!!!!!!!!!!!!!!!!!!!!!!!!!!!!!!!!!!!!!!!!!!!!!!!!!!!!!

\begin{example} \color{black} \label{example1}  % \ref{example1}
Let $ \mathbf{A} \in \mathcal{PCM}_4$ be as follows:
\[
\mathbf{A} =
\begin{pmatrix}
$\,$  1    $\,\,$ & $\,\,$   1     $\,\,$ & $\,\,$   4    $\,\,$ & $\,\,$     7  $\,$    \\
$\,$  1    $\,\,$ & $\,\,$   1     $\,\,$ & $\,\,$   7    $\,\,$ & $\,\,$     4  $\,$    \\
$\,$ 1/4   $\,\,$ & $\,\,$  1/7    $\,\,$ & $\,\,$   1    $\,\,$ & $\,\,$     3  $\,$   \\
$\,$ 1/7   $\,\,$ & $\,\,$  1/4    $\,\,$ & $\,\,$  1/3   $\,\,$ & $\,\,$     1   $\,$
\end{pmatrix}.
\]
The principal right eigenvector $\mathbf{w}^{EM}$ and the consistent
approximation of $\mathbf{A}$ generated by $\mathbf{w}^{EM}$ are as follows:
\[
\mathbf{w}^{EM} =
\begin{pmatrix}
0.39940672 \\
0.43144159 \\
0.10721105 \\
0.06194064
\end{pmatrix}, %\hspace{-1.3cm}
\quad %\qquad
\left[ \frac{w^{EM}_i}{w^{EM}_j} \right] =
\begin{pmatrix}
$\,$ 1        $\,\,$ & $\,\,$ 0.9257   $\,\,$ & $\,\,$ 3.7254   $\,\,$ & $\,\,$ 6.4482 $\,$    \\
$\,$ 1.0802   $\,\,$ & $\,\,$      1   $\,\,$ & $\,\,$ 4.0242   $\,\,$ & $\,\,$ 6.9654 $\,$    \\
$\,$ 0.2684   $\,\,$ & $\,\,$ 0.2485   $\,\,$ & $\,\,$       1  $\,\,$ & $\,\,$ 1.7309 $\,$   \\
$\,$ 0.1551   $\,\,$ & $\,\,$ 0.1436   $\,\,$ & $\,\,$ 0.5777   $\,\,$ & $\,\,$       1  $\,$
\end{pmatrix}.
\]

Apply Def{\kern0pt}inition \ref{def:DirectedGraphEfficient}, the directed graph $G=(V,\overrightarrow{E})_{\mathbf{A},\mathbf{w}^{EM}}$ corresponding to $\mathbf{A}$ and $\mathbf{w}^{EM}$ is drawn in Figure \ref{fig:example}. By Theorem \ref{thm:TheoremDirectedGraphEfficient}, $\mathbf{w}^{EM}$ is not ef{\kern0pt}f{\kern0pt}icient, because the corresponding digraph is not strongly connected: no arc leaves node 1.
\end{example}

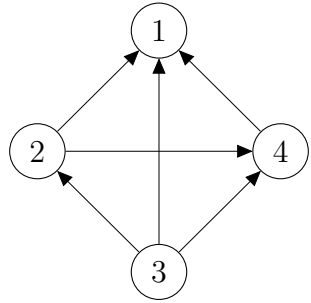
\begin{figure}[h]
\centering
\captionsetup{justification=centering}
\begin{tikzpicture}
\def \radius {1.6cm}

  \node (4) [draw, circle] at ({360/4 * (1 - 1)}:\radius) {$4$};
  \node (1) [draw, circle] at ({360/4 * (2 - 1)}:\radius) {$1$};
  \node (2) [draw, circle] at ({360/4 * (3 - 1)}:\radius) {$2$};
  \node (3) [draw, circle] at ({360/4 * (4 - 1)}:\radius) {$3$};

\draw[->,-,line width=0.25pt,arrows={-triangle 45}] (2) -- (1) node [midway, above, sloped] (TextNode) {};
\draw[->,line width=0.25pt,arrows={-triangle 45}] (3) -- (1) node [midway, above, sloped] (TextNode) {};
\draw[->,line width=0.25pt,arrows={-triangle 45}] (4) -- (1) node [midway, above,sloped] (TextNode) {};
\draw[->,line width=0.25pt,arrows={-triangle 45}] (3) -- (2) node [midway, above, sloped] (TextNode) {};
\draw[->,line width=0.25pt,arrows={-triangle 45}] (2) -- (4) node [midway, above, sloped] (TextNode) {};
\draw[->,line width=0.25pt,arrows={-triangle 45}] (3) -- (4) node [midway, above, sloped] (TextNode) {};

\node [below=2.2cm, align=flush center,text width=8cm] at (1)
        {
        %    $\mbox{Figure 1}$
        };

\end{tikzpicture}
\caption{\small The principal right eigenvector in Example \ref{example1} is inef{\kern0pt}f{\kern0pt}icient, because the corresponding digraph is not strongly connected: no arc leaves node 1}
\label{fig:example}
\end{figure}
%EREDETI ÁBRA!!!!!!!!!!!!!!!!!!!!!!!!!!!!!!!!!!!!!!!!!!!!!!!!!!!!!!!!!!!!!!!!!!!!!!!!!!!!!!!!!!!!!!!!!!!!!
%\begin{figure}
%\centering
%\caption{The principal right eigenvector in Example \ref{example1} is inef{\kern0pt}f{\kern0pt}icient, because the corresponding digraph is not strongly connected}
%\includegraphics[width=20mm]{KDI-example1-telenyil.eps}
%\label{fig:example}
%\end{figure}

%\unitlength 1mm
%\begin{center}
%\begin{picture}(100,40)
%\put(39,10){\resizebox{20mm}{!}{\rotatebox{0}{\includegraphics{KDI-example1.eps}}}}
%\put(-10,3){\makebox{\textbf{Figure 1.}
%The principal right eigenvector in Example \ref{example1} is inef{\kern0pt}f{\kern0pt}icient,}}
%\put(0,-1){\makebox{because the corresponding digraph is not strongly connected }}
%\end{picture}
%\end{center}
% \ref{example1}

\bigskip
\color{black} It can be seen in a more constructive way why the
principal right eigenvector $\mathbf{w}^{EM}$ is
inef{\kern0pt}f{\kern0pt}icient. Increase the f{\kern0pt}irst coordinate
until 0.428844188, and keep the other coordinates unchanged:
\[
\mathbf{w}^{\prime} =
\begin{pmatrix}
0.428844188 \\
0.431441588 \\
0.107211052 \\
0.061940644
\end{pmatrix},
\quad
\left[ \frac{w^{\prime}_i}{w^{\prime}_j} \right] =
\begin{pmatrix}
$\,$ 1        $\,\,$ & $\,\,$ \textbf{0.9940}   $\,\,$ & $\,\,$   \textbf{4}      $\,\,$ & $\,\,$ \textbf{6.9235} $\,$    \\
$\,$ \textbf{1.0061 }  $\,\,$ & $\,\,$      1   $\,\,$ & $\,\,$ 4.0242   $\,\,$ & $\,\,$ 6.9654 $\,$    \\
$\,$ \textbf{1/4 }     $\,\,$ & $\,\,$ 0.2485   $\,\,$ & $\,\,$       1  $\,\,$ & $\,\,$ 1.7309 $\,$   \\
$\,$ \textbf{0.1444}   $\,\,$ & $\,\,$ 0.1436   $\,\,$ & $\,\,$ 0.5777   $\,\,$ & $\,\,$       1  $\,$
\end{pmatrix}.
\]
\color{black}
%[[1.0000000000000000000,0.50548679550734760093,2.5274339775367380046,3.2493478297009048867],
%[1.9782910431840633647,1.0000000000000000000,5.0000000000000000000,6.4281557076868754010],
%[0.39565820863681267294,0.20000000000000000000,1.0000000000000000000,1.2856311415373750802],
%[0.30775406401845496998,0.15556561562505190968,0.77782807812525954840,1.0000000000000000000]]
The approximation in the entries marked by bold became strictly better ((\ref{eqn:eff2}) holds in Def{\kern0pt}inition \ref{def:DefinitionEfficient}), while for all other entries the approximation remained the same ((\ref{eqn:eff1}) holds with equality in Def{\kern0pt}inition \ref{def:DefinitionEfficient}). \\

As it can be seen from Example \ref{example1} above, Theorem \ref{thm:TheoremDirectedGraphEfficient} is a powerful an applicable characterization of ef{\kern0pt}f{\kern0pt}iciency.\\

%PÉLDA EDDIG!!!!!!!!!!!!!!!!!!!!!!!!!!!!!!!!!!!!!!!!!!!!!!!!!!!!!!!!!!!!!!!!!!!!!!!!!!!!!!
\begin{question}
\color{black} What is the necessary and suf{\kern0pt}f{\kern0pt}icient condition of
the principal right eigenvector's ef{\kern0pt}f{\kern0pt}iciency?
\end{question}

In the rest of the paper, special types of PCMs are considered.

A \emph{simple perturbed} PCM dif{\kern0pt}fers from a consistent PCM in only one element and its reciprocal, or in other words it can be made consistent by altering only one element (and its reciprocal). Thus, without loss of generality, a simple perturbed PCM can be written as
\[
\mathbf{A}_{\delta}  =
\left(\begin{array}{ccccc}
1 & \delta x_1 & x_2 & \dots & x_{n-1} \\ 1/(\delta x_1) & 1 & x_2/x_1 & \dots & x_{n-1}/x_1 \\
1/x_2 & x_1/x_2 & 1 & \dots & x_{n-1}/x_2 \\
\vdots & \vdots & \vdots & \ddots & \vdots \\
1/x_{n-1} & x_1/x_{n-1} & x_2/x_{n-1} & \dots & 1
\end{array} \right)\in \mathcal{PCM}_n,
\]
where $x_1,\dots,x_{n-1}>0$ and $0< \delta \neq 1$.

\begin{theorem}[{\cite[Theorem 3.1]{Abele-NagyBozoki2016}}] \label{thm:simple}
The principal right eigenvector of a simple perturbed pairwise comparison matrix is ef{\kern0pt}f{\kern0pt}icient.
\end{theorem}

Similarly, a \emph{double perturbed} PCM dif{\kern0pt}fers from a consistent PCM in two elements and their reciprocals, or in other words it can be made consistent by altering two elements (and their reciprocals). We have to dif{\kern0pt}ferentiate between three cases of double perturbed PCMs. Without loss of generality, every double perturbed PCM is equivalent to one of them. Also, we can suppose without the loss of generality, that from now on $n\geq 4$, because a PCM with $n=3$ is either simple perturbed or consistent. In Case 1, the perturbed elements are in the same row, and they are multiplied by $0<\delta\neq1$ and $0<\gamma\neq1$ respectively. In Case 2, they are in dif{\kern0pt}ferent rows, but this case needs to be further divided into two subcases (2A and 2B) due to algebraic issues. In Case 2A matrix size is $4\times4$, while in Case 2B matrix size is at least $5\times5$. Thus, these matrices take the following form:

Case 1:
\begin{equation}
\label{case1}
\mathbf{P}_{\gamma,\delta} = \begin{pmatrix}
    1 & \delta x_1 & \gamma x_2 & x_3 & \dots & x_{n-1} \\
    1/(\delta x_1) & 1 & x_2/x_1 & x_3/x_1 & \dots & x_{n-1}/x_1 \\
    1/(\gamma x_2) & x_1/x_2 & 1 & x_3/x_2 & \dots & x_{n-1}/x_2 \\
    1/x_3 & x_1/x_3 & x_2/x_3 & 1 & \dots & x_{n-1}/x_3 \\
    \vdots & \vdots & \vdots & \vdots & \ddots & \vdots \\
    1/x_{n-1} & x_1/x_{n-1} & x_2/x_{n-1} & x_3/x_{n-1} & \dots & 1 \end{pmatrix},
\end{equation}

Case 2A:
\begin{equation}
\label{case2}
\mathbf{Q}_{\gamma,\delta}= \begin{pmatrix}
    1 & \delta x_1 & x_2 & x_3 \\
    1/(\delta x_1) & 1 & x_2/x_1 & x_3/x_1 \\
    1/x_2 & x_1/x_2 & 1 & \gamma x_3/x_2 \\
    1/x_3 & x_1/x_3 & x_2/(\gamma x_3) & 1
    \end{pmatrix},
\end{equation}

Case 2B:
\begin{equation}
\label{case3}
\mathbf{R}_{\gamma,\delta}= \begin{pmatrix}
    1 & \delta x_1 & x_2 & x_3 & x_4 & \dots & x_{n-1} \\
    1/(\delta x_1) & 1 & x_2/x_1 & x_3/x_1 & x_4/x_1 & \dots & x_{n-1}/x_1 \\
    1/x_2 & x_1/x_2 & 1 & \gamma x_3/x_2 & x_4/x_2& \dots & x_{n-1}/x_2 \\
    1/x_3 & x_1/x_3 & x_2/(\gamma x_3) & 1 & x_4/x_3 & \dots & x_{n-1}/x_3 \\
    1/x_4 & x_1/x_4 & x_2/x_4 & x_3/x_4 & 1 & \dots & x_{n-1}/x_4 \\
    \vdots & \vdots & \vdots & \vdots & \vdots & \ddots & \vdots \\
    1/x_{n-1} & x_1/x_{n-1} & x_2/x_{n-1} & x_3/x_{n-1} & x_4/x_{n-1} & \dots & 1 \end{pmatrix}.
\end{equation}
Once again, $x_1,\dots,x_{n-1} >0$ and $0<\delta,\gamma \neq 1$.

\begin{remark}
\label{SimpleConsistent}
If either $\delta = 1$ or $\gamma = 1$ then the PCM is simple perturbed. If $\delta = \gamma = 1$ then the PCM is consistent.
\end{remark}

\begin{remark}
If $n=4$ and $\delta = \gamma$, then the PCM $\mathbf{P}_{\delta,\delta}$ in Case 1 is simple perturbed (multiply the single element $x_3$ in position (1,4) by $\delta$ to have a consistent PCM).
\end{remark}

Boz\'oki, F\"ul\"op and Poesz examined PCMs that can be made consistent by modifying at most 3 elements \cite{BozokiFulopPoesz2011}. Each of the three cases above corresponds to a graph: Case 1 corresponds to \cite[Fig. 6(b)]{BozokiFulopPoesz2011} while Case 2 corresponds to \cite[Fig. 6(a)]{BozokiFulopPoesz2011}. Cook and Kress \cite{CookKress1988} and Brunelli and Fedrizzi \cite{BrunelliFedrizzi2015} also examined the similar idea of comparing two PCMs that dif{\kern0pt}fer in only one element.

\section{Main result: \color{black} the principal right eigenvector of a double perturbed PCM is ef{\kern0pt}f{\kern0pt}icient \color{black}}
The main result of the paper is the extension of Theorem \ref{thm:simple} for double perturbed PCMs.

\begin{theorem}
The principal right eigenvector of a double perturbed PCM is ef{\kern0pt}f{\kern0pt}icient.
\end{theorem}
\begin{proof}
For the purpose of easy readability, only an outline of the proof is presented here. The detailed proof can be found in the Appendix.

\color{black} A method to acquire the explicit form of the
principal right eigenvector of a PCM when the perturbed elements
are in the same row or column has been developed by Farkas,
R\'ozsa and Stubnya \cite{FarkasRozsaStubnya1999}. Farkas
\cite{Farkas2007} writes the explicit formula for the simple
perturbed case. Our f{\kern0pt}irst goal is to extend the method for the
double perturbed case. Similar to \cite{Farkas2007}, the
characteristic polynomial is needed f{\kern0pt}irst. \color{black}
Proposition \ref{karpol1} covers Case 1 and Proposition
\ref{karpol2} covers Cases 2A and 2B.

Using the formulas for the characteristic polynomial, explicit formulas can be derived for the principal right eigenvector. Proposition \ref{Formula} presents these formulas. For each Case the formulas can be written in several dif{\kern0pt}ferent forms.

Utilizing the dif{\kern0pt}ferent explicit formulas for the principal right eigenvector a series of inequalities can be proved. These inequalities are presented in 28 lemmas (Lemmas \ref{lemma1a}--\ref{lemma3h}). The dif{\kern0pt}ferent forms of the formula for the principal right eigenvector make it possible to use the form most suited to each proof. Obtaining these inequalities makes it possible to prove the ef{\kern0pt}f{\kern0pt}iciency of the principal right eigenvector of a double perturbed PCM.

As per Theorem \ref{thm:TheoremDirectedGraphEfficient}, the strong connectedness of the digraph in Def{\kern0pt}inition \ref{def:DirectedGraphEfficient} needs to be shown. All possible digraphs are shown in Figures \ref{fig:case1}--\ref{fig:case3}. The direction of each arc (where applicable) is determined by the corresponding Lemma using Def{\kern0pt}inition \ref{def:DirectedGraphEfficient}, which is labeled on the arc itself. In the cases where there is a node named $i$, this represents the complete subgraph of the rest of the nodes (consisting of $n-3$ in Case 1 and $n-4$ nodes in Case 2B). In these subgraphs there are bidirected arcs between any two nodes, due to Lemmas \ref{lemma1j} and \ref{lemma3h}. This is a strongly connected subgraph, and for any f{\kern0pt}ixed $j \leq 3$ the direction of the arc between nodes $i$ and $j$ is the same for every $i \geq 4$ in Case 1 (see Lemmas \ref{lemma1e}, \ref{lemma1f}, \ref{lemma1h}, \ref{lemma1i}). Similarly, for any f{\kern0pt}ixed $j \leq 4$ the direction of the arc between nodes $i$ and $j$ is the same for every $i \geq 5$ in Case 2B (see Lemmas \ref{lemma3c}, \ref{lemma3d}, \ref{lemma3e}). Hence, it can be contracted into a single node when analyzing strong connectedness. Figures \ref{fig:case1}, \ref{fig:case2}, \ref{fig:case3} correspond to Cases 1, 2A, 2B respectively. For the strong connectedness of each digraph, it is suf{\kern0pt}f{\kern0pt}icient to f{\kern0pt}ind a directed cycle. Unchecked arcs are denoted by dashed lines in Figures \ref{fig:case1}--\ref{fig:case3}. The directed cycles are presented in Corollary \ref{cycle1}, \ref{cycle2} and \ref{cycle3} for Cases 1, 2A and 2B respectively.

The presence of a directed cycle implies strong connectedness for all of the digraphs, which implies ef{\kern0pt}f{\kern0pt}iciency in all cases by Theorem \ref{thm:TheoremDirectedGraphEfficient}.
\end{proof}

\begin{figure}[htbp]
\centering
\captionsetup{justification=centering}
\includegraphics[width=130mm]{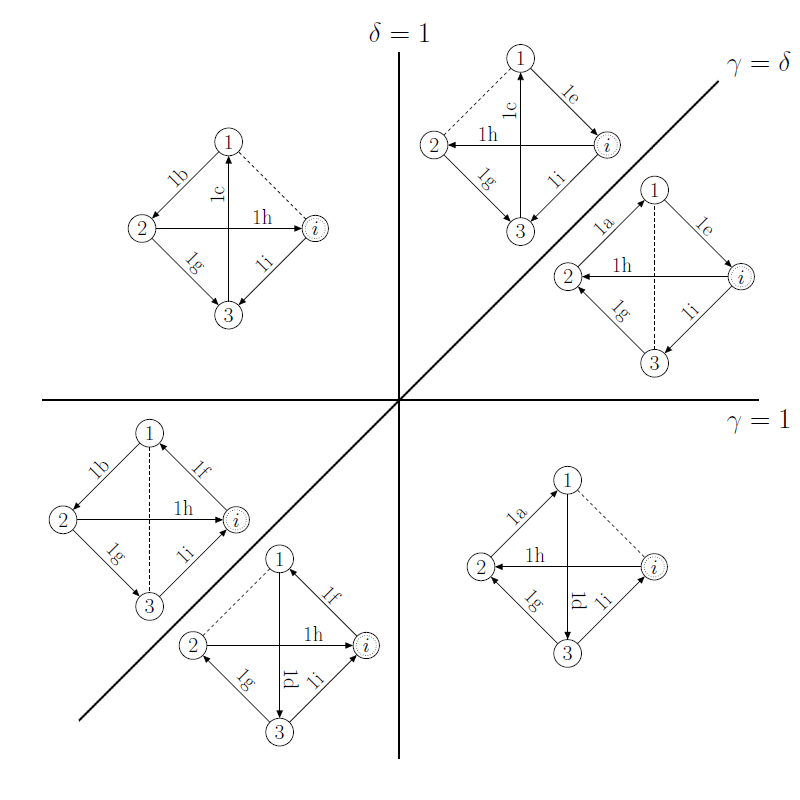}
\caption{\small The digraph of the principal right eigenvector in Case 1 is strongly connected, independently of the orientation of dashed arcs that have not been analyzed}
\label{fig:case1}
\end{figure}
\begin{figure}[htbp]
\centering
\captionsetup{justification=centering}
\includegraphics[width=100mm]{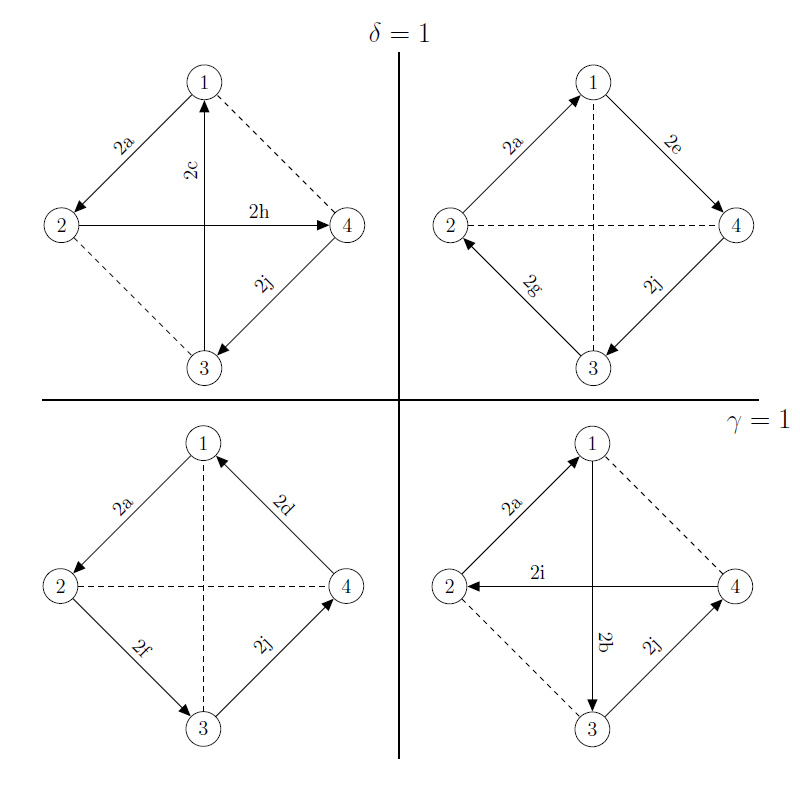}
\caption{\small The digraph of the principal right eigenvector in Case 2A is strongly connected, independently of the orientation of dashed arcs that have not been analyzed}
\label{fig:case2}
\end{figure}
\begin{figure}[htbp]
\centering
\captionsetup{justification=centering}
\includegraphics[width=100mm]{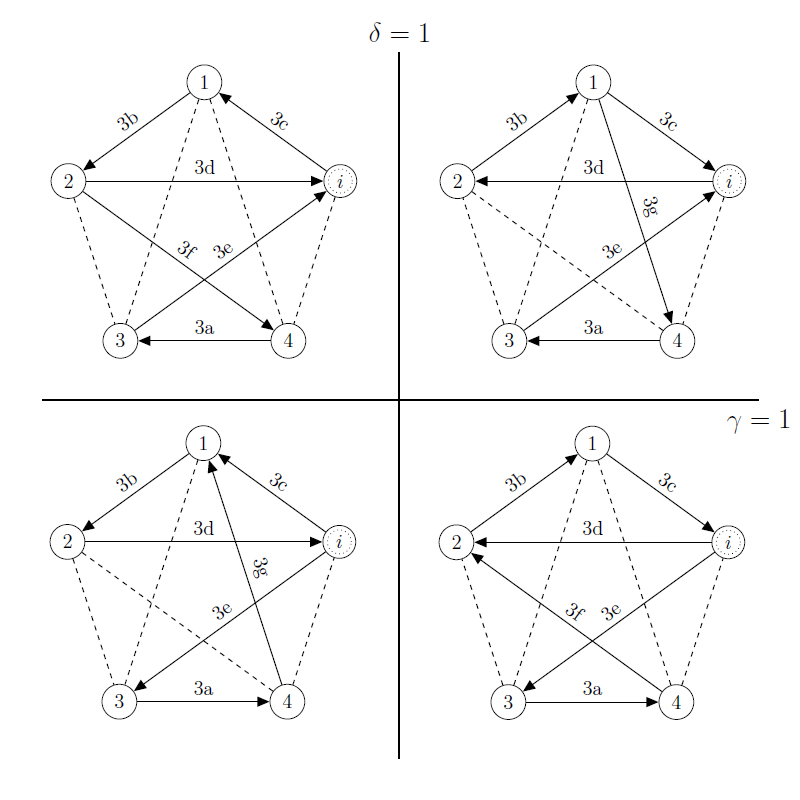}
\caption{\small The digraph of the principal right eigenvector in Case 2B is strongly connected, independently of the orientation of dashed arcs that have not been analyzed}
\label{fig:case3}
\end{figure}

\newpage
\section{Conclusions}
In the paper we used linear algebraic methods to derive explicit formulas for the principal eigenvector of double perturbed PCMs. We also used a necessary and suf{\kern0pt}f{\kern0pt}icient condition for ef{\kern0pt}f{\kern0pt}iciency which uses a directed graph representation (the weight vector is ef{\kern0pt}f{\kern0pt}icient if and only if this graph is strongly connected) developed by Blanquero, Carrizosa and Conde \cite{BlanqueroCarrizosaConde2006}. Double perturbed PCMs had to be  divided into three cases in order to get explicit formulas for every case. In all three cases the digraph has been studied arc by arc, however not all arcs had to be studied in order to determine strong connectedness. Utilizing all these tools, we have shown in the paper, that the often used eigenvector method produces an ef{\kern0pt}f{\kern0pt}icient weight vector in the case of double perturbed PCMs.
This is an extension of our earlier result for simple perturbed PCMs \cite{Abele-NagyBozoki2016}.

\color{black} A direct extension to the triple (or more) perturbed
case is not possible, since all PCMs of at least $4\times 4$ size
which are not (at most) double perturbed are triple perturbed, and
there are examples, e.g. Example 1, of inef{\kern0pt}f{\kern0pt}iciency of size
$4\times 4$. Thus, while in some cases (e.g. when all perturbed
elements are in dif{\kern0pt}ferent rows/columns) it may be possible to show
ef{\kern0pt}f{\kern0pt}iciency, for all triple perturbed PCMs this is impossible.
Furthermore, a triple perturbed PCM can be equivalent to f{\kern0pt}ive
separate basic cases (see \cite[Fig. 7]{BozokiFulopPoesz2011}),
which may need to be further divided into more subcases, making
the ef{\kern0pt}f{\kern0pt}iciency analysis of triple perturbed PCMs dif{\kern0pt}f{\kern0pt}icult. A full
characterization of the ef{\kern0pt}f{\kern0pt}iciency of the principal right
eigenvector is still an open question, and a possible subject of
future research. \color{black}

\section*{Acknowledgements}
The authors are grateful to the anonymous reviewers for their
constructive remarks. Andr\'as Farkas (\'Obuda University, Budapest)
and J\'anos F\"ul\"op (Institute for Computer Science and Control,
Hungarian Academy of Sciences (MTA SZTAKI) and \'Obuda University,
Budapest) are greatly acknowledged for their valuable comments.
Research was supported in part by the Hungarian Scientif{\kern0pt}ic Research Fund (OTKA)
grant K 111797.
S.~Boz\'oki acknowledges the support of the J\'anos Bolyai Research Fellowship no.~BO/00154/16/3.

\bibliographystyle{abbrv}
%\bibliography{References}

\newpage
\section*{Appendix  }
Let $\mathbf{D}=\diag(1,1/x_1,\dots,1/x_{n-1})$, and let $\mathbf{e}=(1,\dots,1)^T$. For $\mathbf{A} \in \mathcal{PCM}_n$,
\begin{equation}
\label{hasonlosag}
\mathbf{ee}^T-\mathbf{U}_i\mathbf{V}_i^T=\mathbf{D}^{-1}\mathbf{AD}
\end{equation}
holds for $i=1,2$. \color{black} For $i=1$, $\mathbf{A}$ has form
(\ref{case1}) (Case 1) and \color{black}
\begin{equation}
\label{UV1}
\mathbf{U}_1=\begin{pmatrix}
0 & 1 \\
1-1/\delta & 0 \\
1-1/\gamma & 0 \\
0 & 0 \\
\vdots & \vdots \\
0 & 0
\end{pmatrix}\in \mathbb{R}^{n \times 2}, \qquad
\mathbf{V}_1=\begin{pmatrix}
1 & 0 \\
0& 1-\delta \\
0 & 1-\gamma \\
0 & 0 \\
\vdots & \vdots \\
0 & 0
\end{pmatrix}\in \mathbb{R}^{n \times 2}.
\end{equation}
\color{black} For $i=2$, $\mathbf{A}$ has form (\ref{case2}) or
(\ref{case3}) (Case 2) and\color{black}
\begin{equation}
\label{UV2}
\mathbf{U}_2=\begin{pmatrix}
0 & 1 & 0 & 0\\
1-1/\delta & 0 & 0 & 0 \\
0 & 0 & 0 & 1 \\
0 & 0 & 1-1/\gamma & 0 \\
0 & 0 & 0 & 0\\
\vdots & \vdots & \vdots & \vdots \\
0 & 0 & 0 & 0
\end{pmatrix} \in \mathbb{R}^{n \times 4},
\mathbf{V}_2=\begin{pmatrix}
1 & 0 & 0 & 0 \\
0& 1-\delta & 0 & 0 \\
0 & 0 & 1 & 0 \\
0 & 0 & 0 & 1-\gamma \\
0 & 0 & 0 & 0\\
\vdots & \vdots & \vdots & \vdots \\
0 & 0 & 0 & 0
\end{pmatrix} \in \mathbb{R}^{n \times 4}.
\end{equation}
%\color{black} if $\mathbf{A}$ has form (\ref{case2}) or (\ref{case3}) (Case 2).\color{black}
%Case 1 with form (\ref{case1}) corresponds to $i=1$, while Case 2 with forms (\ref{case2}) and (\ref{case3}) corresponds to $i=2$.

\begin{lemma}[Matrix determinant lemma, \cite{Harville2008}]
\label{MatDet}
If $\mathbf{A}\in\mathbb{R}^{n\times n}$ is invertible, and $\mathbf{U},\mathbf{V}\in\mathbb{R}^{n\times m}$, then
\[ \det(\mathbf{A}+\mathbf{UV}^T) = \det(\mathbf{I}_m + \mathbf{V}^T\mathbf{A}^{-1}\mathbf{U})\det(\mathbf{A}), \]
where $\mathbf{I}_m$ denotes the identity matrix of size $m \times m$.
\end{lemma}

\begin{lemma}[Sherman--Morrison formula, \cite{Sherman1950}]
\label{SM}
Let $\mathbf{A} \in \mathbb{R}^{n \times n}$, $\mathbf{u},\mathbf{v}\in\mathbb{R}^{n}$. If $\mathbf{A}$ is invertible and $1+\mathbf{v}^T\mathbf{A}^{-1}\mathbf{u} \neq 0$, then $\left(\mathbf{A}+\mathbf{uv}^T\right)^{-1}$ exists, and
$$
\left(\mathbf{A}+\mathbf{uv}^T\right)^{-1} = \mathbf{A}^{-1} - \frac{1}{1+\mathbf{v}^T\mathbf{A}^{-1}\mathbf{u}}\,\mathbf{A}^{-1}\mathbf{uv}^T\mathbf{A}^{-1}.
$$
\end{lemma}

Let $\mathbf{A}\in \mathcal{PCM}_n$ be a double perturbed PCM and $\mathbf{U}_i,\mathbf{V}_i$ be as in (\ref{hasonlosag}). Let the matrix $\mathbf{K}_\mathbf{A}(\lambda) \in \mathbb{R}^{n \times n}$ be def{\kern0pt}ined as follows:
\[ \mathbf{K}_\mathbf{A}(\lambda) = \lambda\mathbf{I}+\mathbf{U}_i\mathbf{V}_i^T-\mathbf{ee}^T = \lambda\mathbf{I} - \mathbf{D}^{-1}\mathbf{AD}, \]
where $\mathbf{I}$ denotes $\mathbf{I}_n$, $\mathbf{e}=(1,\dots,1)^T \in \mathbb{R}^{n}$, $i=1$ in Case 1, $i=2$ in Case 2, and the second equation follows from (\ref{hasonlosag}).

\begin{lemma}
\label{perturbedK}
The characteristic polynomial of the double perturbed PCM $\mathbf{A}\in \mathcal{PCM}_n$ is
\[ p_\mathbf{A}(\lambda)=(-1)^n\det(\mathbf{K}_\mathbf{A}(\lambda)). \]
\end{lemma}
\begin{proof}
As before, $i=1$ in Case 1 and $i=2$ in Case 2.
\begin{align*}
p_\mathbf{A}(\lambda) &= \det(\mathbf{A} - \lambda \mathbf{I}) \\
 &= (-1)^n\det(\lambda \mathbf{I} -\mathbf{A}) \\
 &= (-1)^n\det\left(\lambda \mathbf{I} + \mathbf{D}\left(\mathbf{U}_i\mathbf{V}_i^T - \mathbf{ee}^T\right)\mathbf{D}^{-1}\right) \\
 &= (-1)^n\det\left(\mathbf{D}\left(\lambda \mathbf{I} + \mathbf{U}_i\mathbf{V}_i^T - \mathbf{ee}^T\right)\mathbf{D}^{-1}\right) \\
 &= (-1)^n\det(\mathbf{D})\det\left(\lambda \mathbf{I} + \mathbf{U}_i\mathbf{V}_i^T - \mathbf{ee}^T\right)\det\left(\mathbf{D}^{-1}\right) \\
 &= (-1)^n\det(\mathbf{K}_{\mathbf{A}}(\lambda)).
\end{align*}
\end{proof}

\begin{lemma}
\label{li-eet}
$\det(\lambda\mathbf{I}-\mathbf{ee}^T) = \lambda^n-n\lambda^{n-1}.$
\end{lemma}
\begin{proof}
If $\lambda=0$, then both sides of the equation are $0$.
If $\lambda\neq 0$, apply Lemma \ref{MatDet} with $m=1$, $\mathbf{A}=\lambda \mathbf{I}$, $\mathbf{U}=-\mathbf{e}$, $\mathbf{V}=\mathbf{e}$:
\[
\det(\lambda\mathbf{I}-\mathbf{ee}^T)=(1-\mathbf{e}^T(\lambda\mathbf{I})^{-1}\mathbf{e})\det(\lambda\mathbf{I})=\lambda^n-n\lambda^{n-1}.
\]
\end{proof}

\begin{lemma}
\label{inv}
If $\lambda \neq 0$ and $\lambda \neq n$, then $\left(\lambda \mathbf{I} - \mathbf{ee}^T\right)^{-1}$ exists, and
$$
\left(\lambda \mathbf{I} - \mathbf{ee}^T\right)^{-1} = \frac{1}{\lambda\left(\lambda-n\right)}\mathbf{ee}^T + \frac{1}{\lambda} \mathbf{I}.
$$
\end{lemma}
\begin{proof}
Apply the Sherman--Morrison formula (Lemma \ref{SM}) with $\mathbf{A}=\lambda \mathbf{I}$, $\mathbf{u}=-\mathbf{e}$, $\mathbf{v}=\mathbf{e}$.
\end{proof}

\begin{lemma}
\label{UVVU}
Let $\mathbf{U},\mathbf{V}\in\mathbb{R}^{n\times m}$ be arbitrary matrices. If $\lambda \neq 0$ and $\lambda \neq n$, then
$$\det\left(\lambda \mathbf{I}_n + \mathbf{UV}^T - \mathbf{ee}^T\right) = \left(\lambda^n-n\lambda^{n-1}\right)\det\left(\mathbf{I}_m + \frac{1}{\lambda(\lambda - n)}\mathbf{V}^T\mathbf{ee}^T\mathbf{U} + \frac{1}{\lambda}\mathbf{V}^T\mathbf{U}\right).$$
\end{lemma}
\begin{proof}
Apply Lemma \ref{MatDet} with $\mathbf{A}=\lambda \mathbf{I}_n-\mathbf{ee}^T$. According to Lemma \ref{inv}, $\mathbf{A}$ is invertible. Utilizing Lemmas \ref{MatDet}, \ref{li-eet} and \ref{inv} the following equations hold:
\begin{align*}
& \det\left(\left(\lambda \mathbf{I}_n - \mathbf{ee}^T\right)+\mathbf{UV}^T\right) \\
&= \det\left( \mathbf{I}_m + \mathbf{V}^T\left(\lambda \mathbf{I}_n - \mathbf{ee}^T\right)^{-1}\mathbf{U}\right)\det\left(\lambda \mathbf{I}_n - \mathbf{ee}^T\right) \\
&= \det\left(\mathbf{I}_m + \frac{1}{\lambda(\lambda - n)}\mathbf{V}^T\mathbf{ee}^T\mathbf{U} + \frac{1}{\lambda}\mathbf{V}^T\mathbf{U}\right)\left(\lambda^n-n\lambda^{n-1}\right).
\end{align*}
\end{proof}

We can write the characteristic polynomial of double perturbed PCMs in explicit form.

\begin{proposition}
\label{karpol1}
Let $n \geq 4$. The characteristic polynomial of a double perturbed PCM in form (\ref{case1}) (Case 1) is
$$
p_\mathbf{P}(\lambda) =  (-1)^n\lambda^{n-3} \left( \lambda^3 - n\lambda^2 - \left(\frac{\gamma}{\delta} + \frac{\delta}{\gamma}\right) - (n-3)\left(\gamma+\delta+\frac{1}{\gamma}+\frac{1}{\delta}\right) + 4n -10 \right).
$$
\end{proposition}
\begin{proof}
Lemma \ref{perturbedK} implies that
$$
p_\mathbf{P}(\lambda) = (-1)^n\det(\mathbf{K}_\mathbf{P}(\lambda)) = (-1)^n\det\left(\lambda \mathbf{I} + \mathbf{U}_1{\mathbf{V}_1}^T - \mathbf{ee}^T\right),
$$
where $\mathbf{U}_1$ and $\mathbf{V}_1$ are def{\kern0pt}ined by (\ref{UV1}).
Suppose that $\lambda\neq n$ and $\lambda\neq 0$. According to Lemma \ref{UVVU}
\begin{align*}
p_\mathbf{P}(\lambda) &= (-1)^n\left(\lambda^n-n\lambda^{n-1}\right)\det\left(\mathbf{I}_2 + \frac{1}{\lambda(\lambda - n)}{\mathbf{V}_1}^T\mathbf{ee}^T{\mathbf{U}_1} + \frac{1}{\lambda}{\mathbf{V}_1}^T{\mathbf{U}_1}\right)\\
 &= (-1)^n\left(\lambda^n-n\lambda^{n-1}\right)\det(\mathbf{S})\\
&= (-1)^n\lambda^{n-3} \left( \lambda^3 - n\lambda^2 - \left(\frac{\gamma}{\delta} + \frac{\delta}{\gamma}\right) - (n-3)\left(\gamma+\delta+\frac{1}{\gamma}+\frac{1}{\delta}\right) + 4n -10 \right),
\end{align*}
where
$$
\mathbf{S}=\left(
\begin {array}{cc}
1+\frac{2-1/\delta-1/\gamma}{\lambda(\lambda-n)} & \frac{1}{\lambda(\lambda-n)} + \frac{1}{\lambda}\\
\frac{(2-\delta-\gamma)(1-1/\delta)+(2-\delta-\gamma)(1-1/\gamma)}{\lambda(\lambda-n)}+\frac{(1-\delta)(1-1/\delta)}{\lambda}+\frac{(1-\gamma)(1-1/\gamma)}{\lambda} & 1+\frac{2-\delta-\gamma}{\lambda(\lambda-n)}
\end {array}\right).
$$
A polynomial of degree $n$ is uniquely determined by $n+1$ points, and we have calculated $p_\mathbf{P}(\lambda)$ in all but two points, which completes the proof.
\end{proof}

\begin{proposition}
\label{karpol2}
Let $n \geq 4$. The characteristic polynomial of a double perturbed PCM in form (\ref{case3}) (Case 2B) is
$$
p_\mathbf{R}(\lambda) = (-1)^n \lambda^{n-5}\bigg(\lambda^5-n\lambda^4-(n-2)\left(\gamma+\delta+\frac{1}{\gamma}+\frac{1}{\delta}-4\right)\lambda^2-c\lambda-(n-4)c\bigg),
$$
where
$$
c = \frac{(\gamma-1)^2(\delta-1)^2}{\gamma\delta}.
$$
Furthermore, the characteristic polynomial of a double perturbed PCM in form (\ref{case2}) (Case 2A), $p_\mathbf{Q}(\lambda)$ is a special case of $p_\mathbf{R}(\lambda)$ with $n=4$. Namely,
$$
p_\mathbf{Q}(\lambda) = \lambda^4-4\lambda^3-2\left(\gamma+\delta+\frac{1}{\gamma}+\frac{1}{\delta}-4\right)\lambda-\frac{(\gamma-1)^2(\delta-1)^2}{\gamma\delta}.
$$
\end{proposition}
\begin{proof}
Lemma \ref{perturbedK} implies that
$$
p_\mathbf{R}(\lambda) = (-1)^n\det(\mathbf{K}_\mathbf{R}(\lambda)) = (-1)^n\det\left(\lambda \mathbf{I} + \mathbf{U}_2{\mathbf{V}_2}^T - \mathbf{ee}^T\right),
$$
where $\mathbf{U}_2$ and $\mathbf{V}_2$ are def{\kern0pt}ined by (\ref{UV2}).
Suppose that $\lambda\neq n$ and $\lambda\neq 0$. According to Lemma \ref{UVVU}
\begin{align*}
p_\mathbf{R}(\lambda) &= (-1)^n\left(\lambda^n-n\lambda^{n-1}\right)\det\left(\mathbf{I}_4 + \frac{1}{\lambda(\lambda - n)}{\mathbf{V}_2}^T\mathbf{ee}^T{\mathbf{U}_2} + \frac{1}{\lambda}{\mathbf{V}_2}^T{\mathbf{U}_2}\right)\\
 &= (-1)^n\left(\lambda^n-n\lambda^{n-1}\right)\det(\mathbf{T})\\
 &= (-1)^n \lambda^{n-5}\bigg(\lambda^5-n\lambda^4-(n-2)\left(\gamma+\delta+\frac{1}{\gamma}+\frac{1}{\delta}-4\right)\lambda^2-c\lambda-(n-4)c\bigg),
\end{align*}
where
$$
\mathbf{T}= \left( \begin {array}{cccc}
1+\frac{1-1/\delta}{\lambda(\lambda-n)} & \frac{1}{\lambda(\lambda-n)} + \frac{1}{\lambda} & \frac{1-1/\gamma}{\lambda(\lambda-n)} & \frac{1}{\lambda(\lambda-n)}\\
\frac{(1-\delta)(1-1/\delta)}{\lambda(\lambda-n)} + \frac{(1-\delta)(1-1/\delta)}{\lambda} & 1+\frac{1-\delta}{\lambda(\lambda-n)}& \frac{(1-\delta)(1-1/\gamma)}{\lambda(\lambda-n)} & \frac{1-\delta}{\lambda(\lambda-n)}\\
\frac{1-1/\delta}{\lambda(\lambda-n)} & \frac{1}{\lambda(\lambda-n)} & 1+\frac{1-1/\gamma}{\lambda(\lambda-n)} & \frac{1}{\lambda(\lambda-n)}+\frac{1}{\lambda}\\
\frac{(1-\gamma)(1-1/\delta)}{\lambda(\lambda-n)} & \frac{1-\gamma}{\lambda(\lambda-n)} & \frac{(1-\gamma)(1-1/\gamma)}{\lambda(\lambda-n)} + \frac{(1-\gamma)(1-1/\gamma)}{\lambda} & 1+\frac{1-\gamma}{\lambda(\lambda-n)}
\end {array} \right)
$$
and
$$
c = \frac{(\gamma-1)^2(\delta-1)^2}{\gamma\delta}.
$$
Again, a polynomial of degree $n$ is uniquely determined by $n+1$ points, and we have calculated $p_\mathbf{R}(\lambda)$ in all but two points, which completes the proof.
The case $n=4$ is analogous, and
\[
p_\mathbf{Q}(\lambda) = \lambda^4-4\lambda^3-2\left(\gamma+\delta+\frac{1}{\gamma}+\frac{1}{\delta}-4\right)\lambda-\frac{(\gamma-1)^2(\delta-1)^2}{\gamma\delta}
\]
is resulted in.
\end{proof}

\begin{proposition}
\label{Formula}
The principle right eigenvector of a double perturbed PCM can be written in explicit ways.

In Case 1 ($\gamma$ and $\delta$ are in the same row), the formulas for the principal right eigenvector are the following:
\begin{align}
\label{egysor-1} %(\ref{egysor-1})
\mathbf{w}^{EM} &=
\begin{pmatrix}
  \delta \gamma \lambda (\lambda-n+1)                 \\
\frac{1}{x_1} \left[ \gamma \lambda - (n-2)\gamma + \delta +(n-3)\delta \gamma \right] \\
\frac{1}{x_2} \left[ \delta\lambda  -(n-2)\delta + \gamma +(n-3)\delta\gamma \right] \\
\frac{1}{x_3} \left[ \gamma+\delta+\delta\gamma\lambda-2\delta\gamma  \right] \\
\vdots \\
\frac{1}{x_{i-1}} \left[ \gamma+\delta+\delta\gamma\lambda-2\delta\gamma  \right] \\
\vdots \\
\frac{1}{x_{n-1}} \left[ \gamma+\delta+\delta\gamma\lambda-2\delta\gamma  \right] \\
\end{pmatrix},\\
\label{egysor-2} %(\ref{egysor-2})
\mathbf{w}^{EM} &=
 \begin{pmatrix}
x_1 \gamma\lambda \left[ \delta\lambda-(n-2)\delta+\gamma+n-3 \right] \\
\gamma\lambda^3 - (n-1)\gamma\lambda^2 -(n-3)(\gamma^2-2\gamma+1) \\
\frac{x_1}{x_{2}} \left[
\gamma\lambda^2-\gamma\lambda+\delta\lambda +(n-3)(\delta\gamma-\delta-\gamma+1)
\right] \\
\frac{x_1}{x_{3}} \left[
\gamma\lambda^2-\gamma\lambda-\gamma+\delta+\delta\gamma\lambda-\delta\gamma+\gamma^2
\right] \\
\vdots \\
\frac{x_1}{x_{i-1}} \left[
\gamma\lambda^2-\gamma\lambda-\gamma+\delta+\delta\gamma\lambda-\delta\gamma+\gamma^2
 \right] \\
\vdots \\
\frac{x_1}{x_{n-1}} \left[
\gamma\lambda^2-\gamma\lambda-\gamma+\delta+\delta\gamma\lambda-\delta\gamma+\gamma^2
\right] \\
\end{pmatrix},\\
\label{egysor-3}  %(\ref{egysor-3})
\mathbf{w}^{EM} &=
 \begin{pmatrix}
x_2 \delta\lambda \left[ \delta+\gamma\lambda-(n-2)\gamma+n-3 \right] \\
\frac{x_2}{x_{1}} \left[
\delta\lambda^2-\delta\lambda+\gamma\lambda +(n-3)(\delta\gamma-\delta-\gamma+1)
\right] \\
\delta\lambda^3 -(n-1)\delta\lambda^2 -(n-3)(\delta^2 - 2\delta + 1) \\
\frac{x_2}{x_{3}} \left[
\delta\lambda^2-\delta\lambda+\gamma-\delta+\delta^2+\delta\gamma\lambda-\delta\gamma
\right] \\
\vdots \\
\frac{x_2}{x_{i-1}} \left[
\delta\lambda^2-\delta\lambda+\gamma-\delta+\delta^2+\delta\gamma\lambda-\delta\gamma
 \right] \\
\vdots \\
\frac{x_2}{x_{n-1}} \left[
\delta\lambda^2-\delta\lambda+\gamma-\delta+\delta^2+\delta\gamma\lambda-\delta\gamma
\right] \\
\end{pmatrix},\\
\label{egysor-4} %(\ref{egysor-4})
\mathbf{w}^{EM} &=
 \begin{pmatrix}
x_3 \delta\gamma\lambda (\delta+\gamma+\lambda-2)  \\
\frac{x_3}{x_{1}} \left[
\delta\gamma\lambda^2-\delta\gamma\lambda+\gamma^2+\gamma\lambda-\gamma-\delta\gamma+\delta
\right] \\
\frac{x_3}{x_{2}} \left[
\delta\gamma\lambda^2-\delta\gamma\lambda-\delta\gamma+\gamma+\delta^2+\delta\lambda-\delta
\right] \\
\delta\gamma\lambda^2-4\delta\gamma+\gamma+\delta +\delta^2\gamma+\gamma^2\delta \\
\frac{x_3}{x_4} \left[
\delta\gamma\lambda^2-4\delta\gamma+\gamma+\delta +\delta^2\gamma+\gamma^2\delta
 \right] \\
\vdots \\
\frac{x_3}{x_{n-1}} \left[
\delta\gamma\lambda^2-4\delta\gamma+\gamma+\delta +\delta^2\gamma+\gamma^2\delta
\right] \\
\end{pmatrix}.
\end{align}
\color{black}
Formulas (\ref{egysor-1})--(\ref{egysor-4}) give the same principal right eigenvector, up to a scalar multiplier.\\
\color{black}

In Case 2A ($\gamma$ and $\delta$ are in dif{\kern0pt}ferent rows, and matrix size is $4\times4$) the formulas take the following form:
\begin{align}
\label{4x4-1}  %(\ref{4x4-1})
\mathbf{w}^{EM} &=
\begin{pmatrix}
\delta(\lambda^3\gamma -3\lambda^2\gamma -1 +2\gamma -\gamma^2) \\
\frac{1}{x_{1}} \left[
\lambda^2\gamma-2\lambda\gamma+\delta+2\lambda\delta\gamma-2\delta\gamma+\delta\gamma^2
\right] \\
\frac{1}{x_{2}} \gamma \left[
\gamma+\lambda-1+\delta\lambda^2-2\lambda\delta+\delta+\lambda\delta\gamma-\delta\gamma
\right] \\
\frac{1}{x_{3}} \left[
1+\lambda\gamma-\gamma+\lambda\delta-\delta+\delta\gamma\lambda^2-2\lambda\delta\gamma+\delta\gamma
\right]
\end{pmatrix},\\
\label{4x4-2} %(\ref{4x4-2})
\mathbf{w}^{EM} &=
 \begin{pmatrix}
x_{1} [ \delta\gamma\lambda^2-2\lambda\delta\gamma+1+2\lambda\gamma-2\gamma+\gamma^2 ] \\
\lambda^3\gamma-3\lambda^2\gamma-1+2\gamma-\gamma^2 \\
\frac{x_{1}}{x_{2}} \gamma \left[
\lambda\gamma+\lambda^2-2\lambda-\gamma+1+\lambda\delta-\delta+\delta\gamma
\right] \\
\frac{x_{1}}{x_{3}}  \left[
\lambda+\lambda^2\gamma-2\lambda\gamma-1+\gamma+\delta+\lambda\delta\gamma-\delta\gamma
\right]
\end{pmatrix},\\
\label{4x4-3} %(\ref{4x4-3})
\mathbf{w}^{EM} &=
 \begin{pmatrix}
x_2 \delta(1+\lambda\gamma-\gamma)(\delta+\lambda-1)  \\
\frac{x_{2}}{x_{1}}  \left[
1+\lambda\gamma-\gamma+\lambda\delta-\delta+\delta\gamma\lambda^2-2\lambda\delta\gamma+\delta\gamma
\right] \\
\gamma(\delta\lambda^3 - 3\delta\lambda^2 - 1 + 2\delta - \delta^2)  \\
\frac{x_{2}}{x_{3}}  \left[
2\lambda\delta\gamma+\delta\lambda^2-2\lambda\delta-2\delta\gamma+\gamma+\delta^2\gamma
\right]
\end{pmatrix},\\
\label{4x4-4} %(\ref{4x4-4})
\mathbf{w}^{EM} &=
 \begin{pmatrix}
x_3 \delta( \lambda\gamma+\lambda^2-2\lambda-\gamma+1+\lambda\delta-\delta+\delta\gamma  )  \\
\frac{x_{3}}{x_{1}}  \left[
\gamma+\lambda-1+\delta\lambda^2-2\lambda\delta+\delta+\lambda\delta\gamma-\delta\gamma
\right] \\
\frac{x_{3}}{x_{2}}  \left[
2\lambda\delta+\delta\gamma\lambda^2-2\lambda\delta\gamma-2\delta+1+\delta^2
\right] \\
\delta\lambda^3-3\delta\lambda^2-1+2\delta-\delta^2
\end{pmatrix}.
\end{align}
\color{black}
Again, formulas (\ref{4x4-1})--(\ref{4x4-4}) give the same principal right eigenvector, up to a scalar multiplier.\\
\color{black}

In Case 2B ($\gamma$ and $\delta$ are in dif{\kern0pt}ferent rows, and matrix size is at least $5\times5$) the formulas are the following:

\begin{align}
\label{5x5-1} %(\ref{5x5-1})
%\hspace*{-31mm}
\mathbf{w}^{EM} &=
\begin{pmatrix}
\delta\lambda [ \lambda^3\gamma - (n-1)\lambda^2\gamma - (n-3)(\gamma^2 - 2\gamma + 1) ] \\
\scriptstyle \frac{1}{x_1} \left\{ \lambda^3\gamma-(n-2)\lambda^2\gamma
+(n-2)\delta\gamma\lambda^2 + [\lambda\delta + (n-4)(\delta-1)
](\gamma^2- 2\gamma+ 1 )
 \right\}\\
\frac{1}{x_2} \gamma\lambda \left[
\gamma+\lambda-1+\delta\lambda^2
-2\lambda\delta+\delta+\lambda\delta\gamma
-\delta\gamma
 \right] \\
\frac{1}{x_3} \lambda \left[
 1+\lambda\gamma-\gamma+\lambda\delta  -\delta+\delta\gamma\lambda^2  -2\lambda\delta\gamma+\delta\gamma \right] \\
\scriptstyle \frac{1}{x_4} \left[
\gamma^2-2\gamma+\lambda^2\gamma+1+\lambda\delta-\delta\gamma\lambda^2-2\lambda\delta\gamma+\lambda\gamma^2\delta
+\lambda^3\delta\gamma-\delta+2\delta\gamma-\delta\gamma^2   \right] \\
\vdots \\
\scriptstyle \frac{1}{x_{n-1}} \left[
\gamma^2-2\gamma+\lambda^2\gamma+1+\lambda\delta-\delta\gamma\lambda^2-2\lambda\delta\gamma+\lambda\gamma^2\delta
+\lambda^3\delta\gamma-\delta+2\delta\gamma-\delta\gamma^2   \right] \\
\end{pmatrix},\\
\label{5x5-2} %(\ref{5x5-2})
%\hspace*{-31mm}
\mathbf{w}^{EM} &=
 \begin{pmatrix} \scriptstyle
x_1 [ \lambda^3\delta\gamma - (n-2)\delta\gamma\lambda^2
 - (n-4)\delta(\gamma-1)^2
  + \lambda + (n-2)\lambda^2\gamma - 2\lambda\gamma
 + \lambda\gamma^2
 + (n-4)(\gamma-1)^2 ] \\
\lambda(\lambda^3\gamma -
(n-1)\lambda^2\gamma
-(n-3)(\gamma-1)^2 ) \\
\frac{x_1}{x_2}\gamma\lambda
( \lambda\gamma+\lambda^2-2\lambda-\gamma+1+\delta\lambda-\delta+\delta\gamma ) \\
\frac{x_1}{x_3}\lambda
(\lambda+\lambda^2\gamma-2\lambda\gamma-1+\gamma+\delta+\lambda\delta\gamma-\delta\gamma)  \\
\scriptstyle
\frac{x_1}{x_4}
(\lambda\gamma^2-2\lambda\gamma+\lambda^3\gamma +\lambda-\gamma^2+2\gamma-\lambda^2\gamma-1
+\delta-2\delta\gamma+\delta\gamma^2+\delta\gamma\lambda^2) \\
\vdots \\
\scriptstyle
\frac{x_1}{x_{n-1}}
(\lambda\gamma^2-2\lambda\gamma+\lambda^3\gamma
+\lambda-\gamma^2+2\gamma-\lambda^2\gamma-1
+\delta-2\delta\gamma+\delta\gamma^2+\delta\gamma\lambda^2)
\end{pmatrix},\\
\label{5x5-3} %(\ref{5x5-3})
%\hspace*{-31mm}
\mathbf{w}^{EM} &=
 \begin{pmatrix}
x_2 \delta\lambda (1+\lambda\gamma-\gamma)(\delta+\lambda-1)  \\
\frac{x_2}{x_1}
\lambda(1+\lambda\gamma-\gamma)(1+\delta\lambda-\delta)  \\
\gamma\lambda \left[ \lambda^3\delta - (n-1)\delta\lambda^2 -(n-3)(\delta-1)^2 \right] \\
\scriptstyle \frac{x_2}{x_3} [
\delta\lambda^3 - (n-2)\delta\lambda^2 (1- \gamma)
-2\lambda\delta\gamma +  2(n- 4 )\delta(1-\gamma)  +\lambda\gamma
+\delta^2\lambda\gamma +(n-4)(-1+ \gamma - \delta^2 + \delta^2\gamma ) ] \\
\frac{x_2}{x_4}
(1+\lambda\gamma-\gamma)(\delta\lambda^2+1-2\delta+\delta^2) \\
\vdots \\
\frac{x_2}{x_{n-1}}
(1+\lambda\gamma-\gamma)(\delta\lambda^2+1-2\delta+\delta^2)
\end{pmatrix},\\
\label{5x5-4} %(\ref{5x5-4})
%\hspace*{-31mm}
\mathbf{w}^{EM} &=
 \begin{pmatrix}
x_3 \delta\lambda (\lambda\gamma+\lambda^2-2\lambda-\gamma+1+\delta\lambda-\delta+\delta\gamma)
  \\
\frac{x_3}{x_1}
\lambda(\gamma+\lambda-1)(1+\delta\lambda-\delta)
\\
\scriptstyle \frac{x_3}{x_2}
[ \lambda^3\delta\gamma - (n-2)\delta\lambda^2(\gamma - 1) - 2\delta\lambda +
 2(n- 4 ) \delta(\gamma - 1) + \lambda + \delta^2\lambda +(n-4)( 1 - \gamma  + \delta^2 - \delta^2\gamma )]  \\
 \lambda[\delta\lambda^3 - (n-1)\delta\lambda^2  - (n-3)(\delta-1)^2 ]\\
 \scriptstyle
\frac{x_3}{x_4} (
\delta\gamma\lambda^2+\lambda^3\delta-\delta\lambda^2-2\delta\lambda
-2\delta\gamma+2\delta-1+\gamma+\lambda+\delta^2\lambda-\delta^2+\delta^2\gamma ) \\
\vdots \\
\scriptstyle
\frac{x_3}{x_{n-1}}
( \delta\gamma\lambda^2+\lambda^3\delta-\delta\lambda^2-2\delta\lambda
-2\delta\gamma+2\delta-1+\gamma+\lambda+\delta^2\lambda-\delta^2+\delta^2\gamma )
\end{pmatrix},\\
\label{5x5-5} %(\ref{5x5-5})
%\hspace*{-31mm}
\mathbf{w}^{EM} &=
 \begin{pmatrix}
{x_4} \delta\lambda(\gamma^2-2\gamma+\lambda^2\gamma+1)(\delta+\lambda-1) \\
\frac{x_4}{x_1}
\lambda(\gamma^2-2\gamma+\lambda^2\gamma+1)(1+\delta\lambda-\delta)  \\
\scriptstyle \frac{x_4}{x_2}
\gamma\lambda(\delta\gamma\lambda^2+\lambda^3\delta-\delta\lambda^2
-2\delta\lambda-2\delta\gamma+2\delta-1+\gamma+\lambda+\delta^2\lambda-\delta^2+\delta^2\gamma)  \\
\scriptstyle \frac{x_4}{x_3}
\lambda(\delta\lambda^2+\lambda^3\delta\gamma-\delta\gamma\lambda^2
-2\lambda\delta\gamma-2\delta+2\delta\gamma-\gamma+1+\lambda\gamma+\delta^2+\delta^2\lambda\gamma-\delta^2\gamma) \\
(\gamma^2-2\gamma+\lambda^2\gamma+1)(\delta\lambda^2+1-2\delta+\delta^2) \\
\frac{x_4}{x_5}
(\gamma^2-2\gamma+\lambda^2\gamma+1)(\delta\lambda^2+1-2\delta+\delta^2) \\
\vdots \\
\frac{x_4}{x_{n-1}}
(\gamma^2-2\gamma+\lambda^2\gamma+1)(\delta\lambda^2+1-2\delta+\delta^2) \\
\end{pmatrix}.
\end{align}
\color{black}
Again, formulas (\ref{5x5-1})--(\ref{5x5-5}) give the same principal right eigenvector, up to a scalar multiplier.\\
\color{black}
\end{proposition}

\begin{proof}
The proof is similar to that of the eigenvector formulas (24)--(26) in \cite{Farkas2007}.
Let us consider Case 1. Let $\mathbf{D}=\diag(1,1/x_1,\dots,1/x_{n-1})$, and let $\mathbf{K}_{\mathbf{P}}(\lambda)=\lambda \mathbf{I} + \mathbf{U}_1\mathbf{V}_1^{T} - \mathbf{ee}^{T}$, with $\mathbf{U}_1$ and $\mathbf{V}_1$ as def{\kern0pt}ined by (\ref{UV1}). Since $\mathbf{D}$ is invertible, every column of the one rank matrix $\mathbf{D}\adj(\mathbf{K}_{\mathbf{P}}(\lambda_{\max}))\mathbf{D}^{-1}$ is a Perron eigenvector of $\mathbf{P}$.

For Case 2, replace $\mathbf{U}_1$ by $\mathbf{U}_2$ and $\mathbf{V}_1$ by $\mathbf{V}_2$ as def{\kern0pt}ined by (\ref{UV2}).
\end{proof}

%The formulas for the Perron eigenvector have been written in different forms in Proposition \ref{Formula}. Although it is not at all apparent, these forms are equal up to a constant multiplier.

\begin{remark}
\label{AppRemark}
Formulas (\ref{egysor-1})--(\ref{5x5-5}) are positive.
\end{remark}
\begin{proof}
It is suf{\kern0pt}f{\kern0pt}icient to prove the positivity of any arbitrary element of each formula, because the Perron--Frobenius theorem then guarantees the positivity for the vectors as well. The conclusions of the proofs generally follow from $x_i>0 \ $for all $i=1,\dots,n$, $\gamma,\delta>0$ and $\lambda>n\geq 4$ (or $n\geq 5$ in Case 2B). The proof for each formula follows:

Formula (\ref{egysor-1}):
Positivity is apparent for $w^{EM}_1$.

Formula (\ref{egysor-2}):
\begin{align*}
w^{EM}_1 &= x_1 \gamma\lambda \left[ \delta\lambda-(n-2)\delta+\gamma+n-3 \right] \\ &=
x_1 \gamma\lambda \left[ \delta(\lambda-n+2)+\gamma+(n-3) \right].
\end{align*}

Formula (\ref{egysor-3}):
\begin{align*}
w^{EM}_1 &= x_2 \delta\lambda \left[ \delta+\gamma\lambda-(n-2)\gamma+n-3 \right] \\ &=
x_2 \delta\lambda \left[ \delta+\gamma(\lambda-n+2)+(n-3) \right].
\end{align*}

Formula (\ref{egysor-4}):
Positivity is apparent for $w^{EM}_1$.

Formula (\ref{4x4-1}):
\begin{align*}
w^{EM}_2 &= \frac{1}{x_{1}} \left[ \lambda^2\gamma-2\lambda\gamma+\delta+2\lambda\delta\gamma-2\delta\gamma+\delta\gamma^2 \right] \\ &=
\frac{1}{x_{1}} \left[ \lambda\gamma(\lambda-2)+\delta+2\delta\gamma(\lambda-1)+\delta\gamma^2 \right].
\end{align*}

Formula (\ref{4x4-2}):
\begin{align*}
w^{EM}_1 &= x_{1} [ \delta\gamma\lambda^2-2\lambda\delta\gamma+1+2\lambda\gamma-2\gamma+\gamma^2 ] \\ &=
x_{1} [ \delta\gamma\lambda(\lambda-2)+1+2\gamma(\lambda-1)+\gamma^2 ].
\end{align*}

Formula (\ref{4x4-3}):
\begin{align*}
w^{EM}_1 &= x_2 \delta(1+\lambda\gamma-\gamma)(\delta+\lambda-1) \\ &=
x_2 \delta[1+\gamma(\lambda-1)][\delta+(\lambda-1)].
\end{align*}

Formula (\ref{4x4-4}):
\begin{align*}
w^{EM}_3 &= \frac{x_{3}}{x_{2}}  \left[ 2\lambda\delta+\delta\gamma\lambda^2-2\lambda\delta\gamma-2\delta+1+\delta^2 \right] \\ &= \frac{x_{3}}{x_{2}}  \left[ 2\delta(\lambda-1)+\delta\gamma\lambda(\lambda-2)+1+\delta^2 \right].
\end{align*}

From here on in the proof, $n\geq 5$.

Formula (\ref{5x5-1}):
$w^{EM}_3$ in formula (\ref{5x5-1}) is the same as $\lambda w^{EM}_3$ in formula (\ref{4x4-1}), which is already proven to be positive.

Formula (\ref{5x5-2}):
$w^{EM}_3$ in formula (\ref{5x5-2}) is the same as $\lambda w^{EM}_3$ in formula (\ref{4x4-2}), which is already proven to be positive.

Formula (\ref{5x5-3}):
\begin{align*}
w^{EM}_1 &= x_2 \delta\lambda (1+\lambda\gamma-\gamma)(\delta+\lambda-1) \\ &=
x_2 \delta\lambda [1+\gamma(\lambda-1)][\delta+(\lambda-1)].
\end{align*}

Formula (\ref{5x5-4}):
\begin{align*}
w^{EM}_2 &= \frac{x_3}{x_1} \lambda(\gamma+\lambda-1)(1+\delta\lambda-\delta) \\ &=
\frac{x_3}{x_1} \lambda[\gamma+(\lambda-1)][1+\delta(\lambda-1)].
\end{align*}

Formula (\ref{5x5-5}):
\begin{align*}
w^{EM}_1 &= {x_4} \delta\lambda(\gamma^2-2\gamma+\lambda^2\gamma+1)(\delta+\lambda-1) \\ &=
{x_4} \delta\lambda[\gamma^2+\gamma(\lambda^2-2)+1][\delta+(\lambda-1)].
\end{align*}
\end{proof}

Using these formulas, the paper's main result can be obtained through a series of lemmas. Each of these lemmas corresponds to a directed edge in a digraph. Using these results, the direction of certain arcs can be determined. Thus, it will be shown that directed graphs of Cases 1, 2A and 2B are strongly connected. By Theorem \ref{thm:TheoremDirectedGraphEfficient}, ef{\kern0pt}f{\kern0pt}iciency of the principal right eigenvector is implied.

It follows from the positivity of $\textbf{w}^{EM}$ (see Remark \ref{AppRemark}), that both sides of the starting inequalities of each lemma can be multiplied by the respective $w^{EM}_i$ without further discussion. Since there are 28 lemmas, the proofs are in the Appendix.

Cases of $\delta=1$ and $\gamma=1$ are not covered by Lemmas \ref{lemma1a}--\ref{lemma3h} due to Remark \ref{SimpleConsistent}.

The f{\kern0pt}irst group of lemmas correspond to Case 1 ($\gamma$ and $\delta$ are in the same row), i.e., the double perturbed PCM is written in form (\ref{case1}).

\begin{customlemma}{1a}[Case 1]
\label{lemma1a}
$\delta > 1$ and $\delta \geq \gamma \Rightarrow w_1^{EM}/w_2^{EM} < \delta x_1$.
\end{customlemma}
\begin{proof}
Using formula (\ref{egysor-2}),
$$
\frac{w_1^{EM}}{w_2^{EM}} = x_1 \frac{\gamma\lambda \left( \delta\lambda-(n-2)\delta+\gamma+n-3 \right)}{\gamma\lambda^3 - (n-1)\gamma\lambda^2 -(n-3)(\gamma^2-2\gamma+1)}.
$$
Substitute $\lambda = \lambda_{\max}$ in the characteristic polynomial $p_\mathbf{P}(\lambda)$ by Proposition \ref{karpol1}:
\[
(-1)^n\lambda^{n-3} \left( \lambda^3 - n\lambda^2 - \left(\frac{\gamma}{\delta} + \frac{\delta}{\gamma}\right) - (n-3)\left(\gamma+\delta+\frac{1}{\gamma}+\frac{1}{\delta}\right) + 4n -10 \right) = 0,
\]
which can be transformed to
\begin{equation}
\label{lemma1aeq}
\gamma\delta \lambda^3 - \gamma\delta n \lambda^2 = \gamma^2 + \delta^2 + (n-3)\left(\gamma^2\delta+\gamma\delta^2  + \delta + \gamma\right)-\gamma\delta(4n-10).
\end{equation}
The statement to be proven is equivalent to
$$
\gamma\lambda\left(\delta\lambda+\gamma-(n-2)\delta+n-3\right) < \delta\left(\gamma\lambda^3-(n-1)\gamma\lambda^2-(n-3)(\gamma-1)^2\right).
$$
Using (\ref{lemma1aeq}) this is further equivalent to
\[
(\gamma\delta\lambda^3-\gamma\delta n \lambda^2) + \lambda\left(\gamma\delta(n-2)-\gamma^2-\gamma n + 3 \gamma \right) - \delta (n-3)(\gamma - 1)^2 > 0.
\]
Now apply further equivalent transformations:
\begin{multline*}
\gamma\lambda\left(\delta(n-2)-\gamma-n+3\right) +
\gamma^2 + \delta^2 - (n-3)(\delta\gamma^2-2\delta\gamma + \delta)  \\ + (n-3)\left(\delta^2\gamma+\delta\gamma^2+\delta+\gamma\right)-\delta\gamma(4n-10) > 0
\end{multline*}
\vspace{-1cm}
\begin{multline*}
\gamma\lambda\left((\delta-1)(n-3)+\delta-\gamma\right) +
\gamma^2 + \delta^2 \\ +  (n-3)\left(\delta^2\gamma+2\delta\gamma+\gamma\right)-4\delta\gamma(n-3)-2\delta\gamma > 0
\end{multline*}
\[
\gamma\lambda\left((\delta-1)(n-3)+(\delta-\gamma)\right) + \gamma(n-3)(\delta-1)^2 + (\delta-\gamma)^2 > 0.
\]
\end{proof}

\begin{customlemma}{1b}[Case 1]
\label{lemma1b}
$\delta < 1$ and $\delta \leq \gamma \Rightarrow w_1^{EM}/w_2^{EM} > \delta x_1$.
\end{customlemma}
\begin{proof}
According to formula (\ref{egysor-2})
$$
\frac{w_1^{EM}}{w_2^{EM}} = x_1 \frac{\gamma\lambda \left( \delta\lambda-(n-2)\delta+\gamma+n-3 \right)}{\gamma\lambda^3 - (n-1)\gamma\lambda^2 -(n-3)(\gamma^2-2\gamma+1)}.
$$
Transforming (\ref{lemma1aeq}) similar to Lemma \ref{lemma1a},
\[
\gamma\lambda\left((\delta-1)(n-3)+\delta-\gamma\right)+\gamma(n-3)(\delta-1)^2+(\delta-\gamma)^2<0.
\]
Transforming this further yields
\begin{align*}
\gamma(\delta-1)(n-3)\left(\lambda+(\delta-1)\right)+\gamma\lambda(\delta-\gamma)+(\delta-\gamma)^2 &< 0\\
%\[ \gamma(\delta-1)(n-3)\left(\lambda+(\delta-1)\right)+(\delta - \gamma)(\gamma\lambda+\delta-\gamma) &< 0 \] % ez valszeg nem kel
\gamma(\delta-1)(n-3)\left(\lambda+(\delta-1)\right)+(\delta - \gamma)(\gamma(\lambda-1)+\delta) &< 0.
\end{align*}
\end{proof}

\begin{customlemma}{1c}[Case 1]
$\gamma>1$ and $\gamma \geq \delta \Rightarrow w_1^{EM}/w_3^{EM} < \gamma x_2$.
\end{customlemma}
\begin{proof}
The proof follows from switching the role of $\delta$ and $\gamma$ in the proof of Lemma \ref{lemma1a}.
\end{proof}

\begin{customlemma}{1d}[Case 1]
$\gamma<1$ and $\gamma \leq \delta \Rightarrow w_1^{EM}/w_3^{EM} > \gamma x_2$.
\end{customlemma}
\begin{proof}
The proof follows from switching the role of $\delta$ and $\gamma$ in the proof of Lemma \ref{lemma1b}.
\end{proof}

\begin{customlemma}{1e}[Case 1]
\label{lemma1e}
$\gamma,\delta>1 \Rightarrow w_1^{EM}/w_i^{EM} > x_{i-1}$, $i=4,\dots,n$.
\end{customlemma}
\begin{proof}
According to formula (\ref{egysor-1})
$$
\frac{w_1^{EM}}{w_i^{EM}} = x_{i-1} \frac{\gamma\delta\lambda(\lambda-n+1)}{\gamma+\delta+\gamma\delta\lambda-2\gamma\delta},
$$
which means the statement to be proven is equivalent to
$$
\gamma\delta\lambda(\lambda-n+1) > \gamma+\delta+\gamma\delta\lambda-2\gamma\delta.
$$
Further equivalent transformations yield
\begin{align*}
\left(\gamma\delta\lambda(\lambda-n)\right) + \left(2\gamma\delta - \gamma - \delta \right) &> 0\\
\gamma\delta\lambda(\lambda-n) + (\delta-1)(\gamma-1) + (\delta\gamma - 1) &> 0.
\end{align*}
\end{proof}

\begin{customlemma}{1f}[Case 1]
\label{lemma1f}
$\gamma,\delta<1 \Rightarrow w_1^{EM}/w_i^{EM} < x_{i-1}$, $i=4,\dots,n$.
\end{customlemma}
\begin{proof}
According to formula (\ref{egysor-4})
$$
\frac{w_1^{EM}}{w_i^{EM}} = x_{i-1} \frac{\gamma\delta\lambda (\delta+\gamma+\lambda-2)}{
\gamma\delta\lambda^2-4\gamma\delta+\gamma+\delta +\delta^2\gamma+\gamma^2\delta
 }.
$$
Applying further equivalent transformations
\begin{align*}
\frac{\gamma \delta \lambda (\delta + \gamma + \lambda - 2)}{\gamma \delta \lambda^2 - 4 \gamma \delta + \gamma + \delta + \delta^2 \gamma + \gamma^2 \delta} &< 1 \\
\gamma \delta \lambda (\delta + \gamma + \lambda - 2) &< \gamma \delta (\lambda^2 - 4) + \gamma + \delta + \delta^2 \gamma + \gamma^2 \delta \\
\lambda (\delta + \gamma + \lambda - 2) &< \lambda^2 - 4 + \frac{1}{\delta} + \frac{1}{\gamma} + \delta + \gamma \\
0 &< \lambda^2 - 4 + \frac{1}{\delta} + \frac{1}{\gamma} + \delta + \gamma - \lambda \delta - \lambda \gamma - \lambda^2 + 2 \lambda \\
0 &< 2(\lambda - 2) + (1-\lambda)(\delta + \gamma) + \frac{1}{\delta} + \frac{1}{\gamma} \\
0 &< 2(\lambda - 1) - 2 + (1-\lambda)(\delta + \gamma) + \frac{1}{\delta} + \frac{1}{\gamma} \\
0 &< (\lambda - 1)(2 - \delta - \gamma) + \frac{1}{\delta} + \frac{1}{\gamma} - 2.
\end{align*}
\end{proof}

\begin{customlemma}{1g}[Case 1]
$\delta \lesseqqgtr \gamma \Leftrightarrow w_2^{EM}/w_3^{EM} \gtreqqless x_2/x_1$.
\end{customlemma}
\begin{proof}
According to formula (\ref{egysor-1}), we need to consider
$$
\frac{w_2^{EM}}{w_3^{EM}} = \frac{x_2}{x_1} \cdot \frac{\gamma \lambda - (n-2)\gamma + \delta +(n-3)\gamma \delta}{\delta\lambda -(n-2)\delta + \gamma +(n-3)\gamma \delta} \gtreqqless 1.
$$
Applying further equivalent transformations
\begin{align*}
\gamma \lambda - (n-2)\gamma + \delta +(n-3)\gamma \delta &\gtreqqless \delta\lambda -(n-2)\delta + \gamma +(n-3)\gamma \delta \\
\lambda(\gamma-\delta)-(n-2)(\gamma-\delta)+\delta-\gamma &\gtreqqless 0 \\
(\gamma-\delta)(\lambda-n+1) &\gtreqqless 0.
\end{align*}
The third factor is positive because $\lambda > n$.
\end{proof}

\begin{customlemma}{1h}[Case 1]
\label{lemma1h}
$\delta \gtrless 1 \Leftrightarrow w_2^{EM}/w_i^{EM} \lessgtr x_{i-1}/x_1$, $i=4,\dots,n$.
\end{customlemma}
\begin{proof}
According to formula (\ref{egysor-1})
$$
\frac{w_2^{EM}}{w_i^{EM}} = \frac{x_{i-1}}{x_1} \cdot \frac{\gamma\lambda-(n-2)\gamma+\delta+(n-3)\gamma\delta}{\gamma+\delta+\gamma\delta\lambda-2\gamma\delta}.
$$
Equivalent transformations yield
\begin{align*}
\gamma\lambda-(n-2)\gamma+\delta+(n-3)\gamma\delta &< \gamma+\delta+\gamma\delta\lambda-2\gamma\delta \\
0 &< \gamma(\delta-1)(\lambda-n+1).
\end{align*}
The third factor is positive because $\lambda > n$.
\end{proof}

\begin{customlemma}{1i}[Case 1]
\label{lemma1i}
$\gamma \gtrless 1 \Leftrightarrow w_3^{EM}/w_i^{EM} \lessgtr x_{i-1}/x_2$, $i=4,\dots,n$.
\end{customlemma}
\begin{proof}
The proof follows from switching the role of $\delta$ and $\gamma$ in the proof of Lemma \ref{lemma1h}.
\end{proof}

\begin{customlemma}{1j}[Case 1]
\label{lemma1j}
$w_i^{EM}/w_j^{EM} = x_{j-1}/x_{i-1}$, $i,j=4,\dots,n$.
\end{customlemma}
\begin{proof}
It follows from each of formulas (\ref{egysor-1})--(\ref{egysor-4}).
\end{proof}

\begin{corollary}
\label{cycle1}
There exists a directed cycle in each graph corresponding to Case 1 (Figure \ref{fig:case1}):
\begin{align*}
\delta > 1, \gamma > \delta &: 1 \rightarrow i \rightarrow 2 \rightarrow 3 \rightarrow 1, \\
\gamma > 1, \gamma < \delta &: 1 \rightarrow i \rightarrow 3 \rightarrow 2 \rightarrow 1, \\
\delta > 1, \gamma < 1 &: 1 \rightarrow 3 \rightarrow i \rightarrow 2 \rightarrow 1,\\
\delta < 1, \gamma < \delta &: 1 \rightarrow 3 \rightarrow 2 \rightarrow i \rightarrow 1, \\
\gamma < 1, \gamma > \delta &: 1 \rightarrow 2 \rightarrow 3 \rightarrow i \rightarrow 1, \\
\delta < 1, \gamma > 1 &: 1 \rightarrow 2 \rightarrow i \rightarrow 3 \rightarrow 1.
\end{align*}
\end{corollary}

The second group of lemmas correspond to Case 2A ($\gamma$ and $\delta$ are in dif{\kern0pt}ferent rows, and matrix size is $4\times4$), i.e., the double perturbed PCM is written in form (\ref{case2}).

\begin{customlemma}{2a}[Case 2A]
\label{lemma2a}
$\delta \gtrless 1 \Leftrightarrow w^{EM}_1/w^{EM}_2 \lessgtr \delta x_1$.
\end{customlemma}
\begin{proof}
Formula (\ref{4x4-4}) is used for this proof. Multiplying both sides by $w^{EM}_2$, the statement to be proven %(for $\delta > 1$)
can be written as:
\begin{multline}
\label{ineq2astart}
 x_3 \delta( \lambda\gamma+\lambda^2-2\lambda-\gamma+1+\lambda\delta-\delta+\delta\gamma  ) \\ \lessgtr
\delta x_1 \frac{x_{3}}{x_{1}}  \left(
\gamma+\lambda-1+\delta\lambda^2-2\lambda\delta+\delta+\lambda\delta\gamma-\delta\gamma
\right).
\end{multline}
Further equivalent transformations yield:
\begin{align*}
0 &\lessgtr
\lambda^2 \delta - \lambda^2 + 3\lambda - \lambda \gamma - 3\lambda \delta + \lambda\delta\gamma - 2\delta\gamma + 2 \gamma + 2 \delta - 2 \\
0 &\lessgtr
\lambda^2 (\delta - 1) + \lambda \gamma (\delta - 1) + 3\lambda (1-\delta) + 2 \gamma (1-\delta) + 2 (\delta - 1) \\
0 &\lessgtr
(\delta - 1)(\lambda (\lambda - 3) + \gamma (\lambda - 2) + 2).
\end{align*}
The second factor on the right hand side is always positive because $\lambda > n = 4$ and $\gamma, \delta >0$.
%The expression on the right hand side is positive if and only if $\delta > 1$. The case for $\delta < 1$ is obtained by the observation that the same expression is negative if and only if $\delta < 1$.
\end{proof}

\begin{customlemma}{2b}[Case 2A]
\label{lemma2b}
$\delta > 1, \gamma <1 \Rightarrow w^{EM}_1/w^{EM}_3 > x_2$.
\end{customlemma}
\begin{proof}
Formula (\ref{4x4-2}) is used in this proof. Multiplying both sides by $w^{EM}_3$, the statement of the lemma is equivalent to:
\begin{multline*}
x_{1} ( \delta\gamma\lambda^2-2\lambda\delta\gamma+1+2\lambda\gamma-2\gamma+\gamma^2 ) \\
 < x_2 \frac{x_{1}}{x_{2}} \gamma \left(
\lambda\gamma+\lambda^2-2\lambda-\gamma+1+\lambda\delta-\delta+\delta\gamma
\right).
\end{multline*}
Further equivalent transformations yield:
\begin{align}
0 &<
\lambda^2 \gamma - \lambda^2 \gamma \delta - 4 \lambda \gamma + \lambda \gamma^2 + 3 \lambda \delta \gamma - 2 \gamma^2 + 3\gamma + \delta\gamma^2 - \delta \gamma -1 \notag \\
0 &<
\lambda^2 \gamma (1-\delta) + \lambda \gamma (\gamma-1) + 3\lambda \gamma (\delta-1) + (\gamma-1) + 2\gamma(1-\gamma) + \delta\gamma (\gamma-1) \notag \\
\label{ineq2b}
0 &<
(1-\delta)\lambda\gamma(\lambda-3)+(\gamma-1)(\gamma(\lambda-2)+\delta\gamma+1).
\end{align}
The second factor on the right hand side is always positive because $\lambda > n = 4$ and $\gamma, \delta >0$.
\end{proof}

\begin{customlemma}{2c}[Case 2A]
$\delta < 1, \gamma >1 \Rightarrow w^{EM}_1/w^{EM}_3 < x_2$.
\end{customlemma}
\begin{proof}
The proof follows from the right hand side of (\ref{ineq2b}) being positive in the case of $\delta < 1, \gamma >1$.
\end{proof}

\begin{customlemma}{2d}[Case 2A]
\label{lemma2d}
$\delta,\gamma<1 \Rightarrow w^{EM}_1/w^{EM}_4<x_3$.
\end{customlemma}
\begin{proof}
Again, formula (\ref{4x4-2}) is used for this proof. Multiplying both sides by $w^{EM}_4$, the statement to be proven is equivalent to:
\begin{multline*}
x_{1} ( \delta\gamma\lambda^2-2\lambda\delta\gamma+1+2\lambda\gamma-2\gamma+\gamma^2 ) \\
 < x_3 \frac{x_{1}}{x_{3}}  \left(
\lambda+\lambda^2\gamma-2\lambda\gamma-1+\gamma+\delta+\lambda\delta\gamma-\delta\gamma
\right).
\end{multline*}
Further equivalent transformations yield:
\begin{align}
\lambda^2 \gamma \delta - \lambda^2\gamma + 4 \lambda \gamma - 3 \lambda \delta \gamma - \lambda +\gamma^2 - 3\gamma + \delta\gamma - \delta + 2
&< 0 \notag \\
(\delta-1)(\lambda^2 \gamma - 3 \lambda \gamma) + (\gamma-1)(\lambda + \gamma - 2 + \delta)
&< 0 \notag \\
\label{ineq2d}
(\delta-1)\lambda \gamma (\lambda - 3) + (\gamma-1)((\lambda - 2) + \gamma  + \delta)
&< 0.
\end{align}
The left hand side is negative if $\gamma,\delta<1$, because $\lambda > n = 4$.
\end{proof}

\begin{customlemma}{2e}[Case 2A]
$\delta,\gamma>1 \Rightarrow w^{EM}_1/w^{EM}_4>x_3$.
\end{customlemma}
\begin{proof}
The proof follows from the left hand side of (\ref{ineq2d}) being positive if $\gamma,\delta>1$.
\end{proof}

\begin{customlemma}{2f}[Case 2A]
\label{lemma2f}
$\delta,\gamma<1 \Rightarrow w^{EM}_2/w^{EM}_3>x_2/x_1$.
\end{customlemma}
\begin{proof}
Formula (\ref{4x4-1}) is used in this proof. Multiplying both sides by $w^{EM}_3$, the statement of the lemma can be written as:
\begin{multline*}
\frac{1}{x_{1}} \left(
\lambda^2\gamma-2\lambda\gamma+\delta+2\lambda\delta\gamma-2\delta\gamma+\delta\gamma^2
\right) \\ >
\frac{x_2}{x_1} \frac{1}{x_{2}} \gamma \left(
\gamma+\lambda-1+\delta\lambda^2-2\lambda\delta+\delta+\lambda\delta\gamma-\delta\gamma
\right).
\end{multline*}
Further equivalent transformations yield:
\begin{align}
0 &>
\lambda^2\gamma\delta - \lambda^2\gamma - 4\lambda\delta\gamma + 3\lambda\gamma + \lambda\delta\gamma^2 + \gamma^2 - 2\delta\gamma^2 + 3\delta\gamma - \gamma - \delta \notag \\
0 &>
(\delta-1)(\lambda^2\gamma-3\lambda\gamma)+(\gamma-1)(\lambda\delta\gamma-2\delta\gamma+\delta+\gamma)
\notag \\
\label{ineq2f}
0 &>
(\delta-1)\lambda\gamma(\lambda-3)+(\gamma-1)(\delta\gamma(\lambda-2)+\delta+\gamma).
\end{align}
The right hand side is negative if $\delta,\gamma<1$, because $\lambda > n = 4$.
\end{proof}

\begin{customlemma}{2g}[Case 2A]
$\delta,\gamma>1 \Rightarrow w^{EM}_2/w^{EM}_3<x_2/x_1$.
\end{customlemma}
\begin{proof}
The proof follows from the right hand side of (\ref{ineq2f}) being positive if $\delta,\gamma>1$.
\end{proof}

\begin{customlemma}{2h}[Case 2A]
$\delta<1,\gamma>1 \Rightarrow w^{EM}_2/w^{EM}_4 > x_3/x_1$.
\end{customlemma}
\begin{proof}
Again, formula (\ref{4x4-1}) is used in this proof. Multiplying both sides by $w^{EM}_4$, the statement to be proven is equivalent to:
\begin{multline*}
 \frac{1}{x_{1}} \left(
\lambda^2\gamma-2\lambda\gamma+\delta+2\lambda\delta\gamma-2\delta\gamma+\delta\gamma^2
\right) \\ >
\frac{x_3}{x_1}\frac{1}{x_{3}} \left(
1+\lambda\gamma-\gamma+\lambda\delta-\delta+\delta\gamma\lambda^2-2\lambda\delta\gamma+\delta\gamma
\right).
\end{multline*}
Further equivalent transformations yield:
\begin{align}
0 &>
\lambda^2\gamma\delta - \lambda^2\gamma - 4\lambda\delta\gamma + 3\lambda\gamma + \lambda\delta + 3\delta\gamma - \delta\gamma^2 - 2\delta - \gamma + 1 \notag\\
0 &>
(\delta-1)(\lambda^2\gamma - 3\lambda\gamma) + (1-\gamma)(\lambda\delta - 2\delta + \delta\gamma +1) \notag\\
\label{ineq2h}
0 &>
(\delta-1)\lambda\gamma(\lambda - 3) + (1-\gamma)(\delta(\lambda - 2) + \delta\gamma +1).
\end{align}
The right hand side of (\ref{ineq2h}) is negative, if $\delta<1,\gamma>1$, because $\lambda > n = 4$.
\end{proof}

\begin{customlemma}{2i}[Case 2A]
$\delta>1,\gamma<1 \Rightarrow w^{EM}_2/w^{EM}_4 < x_3/x_1$.
\end{customlemma}
\begin{proof}
The proof follows from the right hand side of (\ref{ineq2h}) being positive if $\delta>1,\gamma<1$.
\end{proof}

\begin{customlemma}{2j}[Case 2A]
\label{lemma2j}
$\gamma \gtrless 1 \Leftrightarrow w^{EM}_3/w^{EM}_4 \lessgtr \gamma x_3/x_2$.
\end{customlemma}
\begin{proof}
Once again, formula (\ref{4x4-1}) is used for the proof. Multiplying both sides by $w^{EM}_4$, the f{\kern0pt}irst statement (for $\gamma>1$) becomes equivalent to:
\begin{multline}
\label{ineq2jstart}
\frac{1}{x_{2}} \gamma \left(
\gamma+\lambda-1+\delta\lambda^2-2\lambda\delta+\delta+\lambda\delta\gamma-\delta\gamma
\right) \\
\lessgtr \gamma\frac{x_3}{x_2}\frac{1}{x_{3}} \left(
1+\lambda\gamma-\gamma+\lambda\delta-\delta+\delta\gamma\lambda^2-2\lambda\delta\gamma+\delta\gamma
\right).
\end{multline}
Applying further equivalent transformations:
\begin{align}
0 &\lessgtr
\lambda^2\delta\gamma - \lambda^2\delta - 3\lambda\delta\gamma + 3\lambda\delta + \lambda\gamma - \lambda + 2\delta\gamma - 2\delta - 2\gamma + 2 \notag\\
0 &\lessgtr
(\gamma-1)(\lambda^2\delta - 3\lambda\delta + \lambda + 2\delta -2) \notag\\
\label{ineq2j}
0 &\lessgtr
(\gamma-1)(\lambda\delta(\lambda - 3) + (\lambda -2) + 2\delta).
\end{align}
The second factor on the right hand side of (\ref{ineq2j}) is positive because $\lambda > n = 4$ and $\gamma, \delta >0$.
%if and only if $\gamma>1$. The same expression is negative if and only if $\gamma<1$, thus the second statement can be proven similarly.
\end{proof}

\begin{corollary}
\label{cycle2}
There exists a directed cycle in each graph corresponding to Case 2A (Figure \ref{fig:case2}):
\begin{align*}
\delta > 1, \gamma > 1  &: 1 \rightarrow 4 \rightarrow 3 \rightarrow 2 \rightarrow 1, \\
\delta > 1, \gamma < 1 &: 1 \rightarrow 3 \rightarrow 4 \rightarrow 2 \rightarrow 1, \\
\delta < 1, \gamma < 1 &: 1 \rightarrow 2 \rightarrow 3 \rightarrow 4 \rightarrow 1, \\
\delta < 1, \gamma > 1 &: 1 \rightarrow 2 \rightarrow 4 \rightarrow 3 \rightarrow 1.
\end{align*}
\end{corollary}

The last group of lemmas correspond to Case 2B, when $\gamma$ and $\delta$ are in dif{\kern0pt}ferent rows, and matrix size is at least $5\times5$, i.e., the double perturbed PCM is written in form (\ref{case3}).

\begin{customlemma}{3a}[Case 2B]
$\gamma \gtrless 1 \Leftrightarrow w^{EM}_3/w^{EM}_4 \lessgtr \gamma x_3/x_2$.
\end{customlemma}
\begin{proof}
Using formula (\ref{5x5-1}) the proof is similar to the proof of Lemma \ref{lemma2j}, the only dif{\kern0pt}ference is in (\ref{ineq2jstart}) where both sides are multiplied by $\lambda$, which immediately cancel each other.
\end{proof}

\begin{customlemma}{3b}[Case 2B]
$\delta \gtrless 1 \Leftrightarrow w^{EM}_1/w^{EM}_2 \lessgtr \delta x_1$.
\end{customlemma}
\begin{proof}
Using formula (\ref{5x5-5}) the proof is similar to the proof of Lemma \ref{lemma2a}, the only dif{\kern0pt}ference is in (\ref{ineq2astart}) where both sides (the formula for $w^{EM}_1$ and $w^{EM}_2$) are multiplied by $\lambda$, which immediately cancel each other. This may not be apparent about $w^{EM}_2$, but
\[(\gamma+\lambda-1)(1+\delta\lambda-\delta) = \gamma+\lambda-1 + \lambda\gamma\delta + \lambda^2 \delta - \lambda \delta - \gamma \delta - \lambda \delta + \delta \]
which, after reduction, gives the same formula.
\end{proof}

\begin{customlemma}{3c}[Case 2B]
\label{lemma3c}
$\delta \gtrless 1 \Leftrightarrow w^{EM}_1/w^{EM}_i \gtrless x_{i-1}$,  $i=5,\dots,n$.
\end{customlemma}
\begin{proof}
Formula (\ref{5x5-3}) is used for this proof.
\begin{align*}
x_2 \delta\lambda (1+\lambda\gamma-\gamma)(\delta+\lambda-1) &\gtrless
x_{i-1} \frac{x_2}{x_{i-1}}
(1+\lambda\gamma-\gamma)(\delta\lambda^2+1-2\delta+\delta^2) \\
 \lambda\delta^2 + \lambda^2 \delta - \lambda \delta &\gtrless \delta\lambda^2+1-2\delta+\delta^2 \\
\lambda\delta (\delta - 1) + \delta (1-\delta) + (\delta - 1) &\gtrless 0 \\
(\delta - 1)(\delta(\lambda - 1)+1) &\gtrless 0.
\end{align*}
The second factor on the left hand side is always positive because $\lambda > n \geq 5$ and $\delta >0$.
\end{proof}

\begin{customlemma}{3d}[Case 2B]
\label{lemma3d}
$\delta \gtrless 1 \Leftrightarrow w^{EM}_2/w^{EM}_i \lessgtr x_{i-1}/x_1$, $i=5,\dots,n$.
\end{customlemma}
\begin{proof}
Again, formula (\ref{5x5-3}) is used in the proof.
\begin{align*}
\frac{x_2}{x_1}
\lambda(1+\lambda\gamma-\gamma)(1+\delta\lambda-\delta) &\lessgtr
\frac{x_{i-1}}{x_1} \frac{x_2}{x_{i-1}}
(1+\lambda\gamma-\gamma)(\delta\lambda^2+1-2\delta+\delta^2) \\
 \lambda + \lambda^2 \delta - \delta \lambda &\lessgtr \delta\lambda^2+1-2\delta+\delta^2 \\
0 &\lessgtr
\lambda \delta - \lambda + \delta^2 - 2\delta + 1 \\
0 &\lessgtr \lambda (\delta - 1) + (\delta - 1)^2 \\
0 &\lessgtr (\delta - 1)((\lambda - 1)+\delta).
\end{align*}
The second factor on the right hand side is always positive because $\lambda > n \geq 5$ and $\delta >0$.
\end{proof}

\begin{customlemma}{3e}[Case 2B]
\label{lemma3e}
$\gamma \gtrless 1 \Leftrightarrow w^{EM}_3/w^{EM}_i \gtrless x_{i-1}/x_2$, $i=5,\dots,n$.
\end{customlemma}
\begin{proof}
Formula (\ref{5x5-2}) is used in this proof.
\begin{multline*}\frac{x_1}{x_2}\gamma\lambda
( \lambda\gamma+\lambda^2-2\lambda-\gamma+1+\delta\lambda-\delta+\delta\gamma ) \\ \gtrless
\frac{x_4}{x_2} \frac{x_1}{x_4}
(\lambda\gamma^2-2\lambda\gamma+\lambda^3\gamma +\lambda-\gamma^2+2\gamma-\lambda^2\gamma-1
+\delta-2\delta\gamma+\delta\gamma^2+\delta\gamma\lambda^2).
\end{multline*}
Further equivalent transformations yield:
\begin{align*}
\lambda^2 \gamma^2 - \lambda ^2 \gamma - 2 \lambda \gamma^2 + 3\lambda\gamma + \lambda\gamma^2 \delta - \lambda \delta \gamma - \lambda + \gamma^2 - 2\gamma + 1 - \delta + 2\delta
 \gamma - \delta \gamma^2 &\gtrless 0 \\
(\gamma-1)(\lambda^2\gamma - 2\lambda\gamma + \lambda + \lambda\delta\gamma + (\gamma -1) - \delta(\gamma-1)) &\gtrless 0  \\
(\gamma-1)(\lambda\gamma(\lambda-2) + (\lambda-1) + \delta\gamma(\lambda-1) + \gamma + \delta) &\gtrless 0.
\end{align*}
The second factor on the left hand side is always positive because $\lambda > n \geq 5$ and $\gamma, \delta >0$.
\end{proof}

\begin{customlemma}{3f} \label{lemma3f} ~

$\gamma > 1, \delta < 1 \Rightarrow w^{EM}_2/w^{EM}_4>x_3/x_1$.

$\gamma < 1, \delta > 1 \Rightarrow w^{EM}_2/w^{EM}_4<x_3/x_1$.
\end{customlemma}
\begin{proof}
Instead of the statement of the lemma, we will prove the following stronger statement:

\[\gamma \gtreqqless \delta \Leftrightarrow w^{EM}_2/w^{EM}_4 \gtreqqless x_3/x_1.\]

Formula (\ref{5x5-5}) is used in this proof.
\begin{multline*} \frac{x_4}{x_1}
\lambda(\gamma^2-2\gamma+\lambda^2\gamma+1)(1+\delta\lambda-\delta)
\\ \gtreqqless \frac{x_3}{x_1} \frac{x_4}{x_3}
\lambda(\delta\lambda^2+\lambda^3\delta\gamma-\delta\gamma\lambda^2
-2\lambda\delta\gamma-2\delta+2\delta\gamma-\gamma+1+\lambda\gamma+\delta^2+\delta^2\lambda\gamma-\delta^2\gamma). \end{multline*}
%\[ \Updownarrow \]
%\[ \Leftrightarrow \]
This is further equivalent to
\begin{multline*}\gamma^2 + \lambda\gamma^2\delta-\gamma^2\delta-2\gamma-2\lambda\gamma\delta + 2\gamma\delta + \lambda^2\gamma + \lambda^3\gamma\delta - \lambda^2\gamma\delta + 1 + \lambda\delta - \delta
\\ \gtreqqless \lambda^2\delta + \lambda^3\gamma\delta - \lambda^2\gamma\delta -2\lambda\gamma\delta - 2\delta + 2\gamma\delta - \gamma + 1 + \lambda\gamma + \delta^2 + \lambda\gamma\delta^2 - \gamma\delta^2. \end{multline*}
%\[ \Updownarrow \]
%\[ \Leftrightarrow \]
Further equivalent transformations yield
\[\lambda^2\gamma - \lambda^2\delta + \lambda\delta - \lambda\gamma + \lambda\gamma^2\delta - \lambda\gamma\delta^2 + \gamma^2 - \delta^2 + \gamma\delta^2 - \gamma^2\delta + 2\delta - 2\gamma + \gamma - \delta \gtreqqless 0 \]
%\[ \Updownarrow \]
%\[ \Leftrightarrow \]
\[\lambda^2(\gamma-\delta) + \lambda(\delta-\gamma) + \lambda\gamma\delta(\gamma-\delta) + (\gamma+\delta)(\gamma-\delta) + \gamma\delta(\delta-\gamma) + 2(\delta-\gamma) + (\gamma-\delta) \gtreqqless 0\]
%\[ \Updownarrow \]
%\[ \Leftrightarrow \]
\vspace{-0.6cm}
\begin{align*}
(\gamma-\delta)(\lambda^2 - \lambda + \lambda\gamma\delta + \gamma + \delta - \gamma\delta - 1) &\gtreqqless 0\\
%\[ \Updownarrow \]
%\[ \Leftrightarrow \]
(\gamma-\delta)(\lambda^2 - 2\lambda + \lambda\gamma\delta - \gamma\delta + \lambda - 1 + \gamma + \delta) &\gtreqqless 0 \\
%\[ \Updownarrow \]
%\[ \Leftrightarrow \]
(\gamma-\delta)(\lambda(\lambda - 2) + \gamma\delta(\lambda - 1) + (\lambda - 1) + \gamma + \delta) &\gtreqqless 0.
\end{align*}
The second factor on the left hand side is always positive because $\lambda > n \geq 5$ and $\gamma, \delta >0$.
\end{proof}

\begin{customlemma}{3g}[Case 2B] \label{lemma3g} ~

$\gamma, \delta > 1 \Rightarrow w^{EM}_1/w^{EM}_4>x_3$.

$\gamma, \delta < 1 \Rightarrow w^{EM}_1/w^{EM}_4<x_3$.
\end{customlemma}
\begin{proof}
Instead of the above statement, we will prove the following stronger statement:
\[\gamma\delta \gtreqqless 1 \Leftrightarrow w^{EM}_1/w^{EM}_4 \gtreqqless x_3.\]
Formula (\ref{5x5-5}) is used in this proof.
\begin{multline*}
{x_4} \delta\lambda(\gamma^2-2\gamma+\lambda^2\gamma+1)(\delta+\lambda-1) \\ \gtreqqless x_3 \frac{x_4}{x_3}
\lambda(\delta\lambda^2+\lambda^3\delta\gamma-\delta\gamma\lambda^2
-2\lambda\delta\gamma-2\delta+2\delta\gamma-\gamma+1+\lambda\gamma+\delta^2+\delta^2\lambda\gamma-\delta^2\gamma). \end{multline*}
Further equivalent transformations yield:
\begin{align*}
\lambda^2\delta^2\gamma - \lambda^2\delta + \lambda\gamma^2\delta - \lambda\gamma\delta^2 + \lambda\delta - \lambda\gamma + \gamma^2\delta^2 - \gamma\delta^2 - \delta\gamma^2 + \delta + \gamma - 1 &\gtreqqless 0 \\
(\delta\gamma - 1)(\lambda^2\delta + \lambda\gamma - \lambda\delta + \delta\gamma + 1 - \delta - \gamma) &\gtreqqless 0 \\
(\delta\gamma - 1)(\lambda\delta(\lambda-2) + \gamma(\lambda-1) + \delta(\lambda-1) + \delta\gamma + 1) &\gtreqqless 0.
\end{align*}
The second factor is always positive because $\lambda > n \geq 5$ and $\gamma, \delta >0$. The f{\kern0pt}irst factor is positive exactly if $\gamma\delta > 1$, and negative exactly if $\gamma\delta < 1$.
\end{proof}

\begin{customlemma}{3h}[Case 2B]
\label{lemma3h}
$w_i^{EM}/w_j^{EM} = x_{j-1}/x_{i-1}$, $i,j=5,\dots,n$.
\end{customlemma}
\begin{proof}
It follows from each of formulas (\ref{5x5-1})--(\ref{5x5-5}).
\end{proof}

\begin{corollary}
\label{cycle3}
There exists a directed cycle in each graph corresponding to Case 2B (Figure \ref{fig:case3}):
\begin{align*}
\delta > 1, \gamma > 1  &: 1 \rightarrow 4 \rightarrow 3 \rightarrow i \rightarrow 2 \rightarrow 1, \\
\delta > 1, \gamma < 1 &: 1 \rightarrow i \rightarrow 3 \rightarrow 4 \rightarrow 2 \rightarrow 1, \\
\delta < 1, \gamma < 1 &: 1 \rightarrow 2 \rightarrow i \rightarrow 3 \rightarrow 4 \rightarrow 1, \\
\delta < 1, \gamma > 1 &: 1 \rightarrow 2 \rightarrow 4 \rightarrow 3 \rightarrow i \rightarrow 1.
\end{align*}
\end{corollary}

\end{document}